\documentclass[a4paper,fleqn]{cas-sc}

\usepackage{amssymb}
\usepackage{amsthm}
\usepackage{amsmath}
\usepackage{stmaryrd}
\usepackage{color}

\usepackage{subfig}
\usepackage{tikz,pgfplots}
\pgfplotsset{compat=newest}

\usepackage{algorithm}
\usepackage{algorithmic}
\usepackage{multicol}
\usepackage{tabularx}

\newtheorem{theorem}{Theorem}
\newtheorem{lemma}[theorem]{Lemma}

\newtheorem{remark}[theorem]{Remark}

\newtheorem{corollary}[theorem]{Corollary}

\newcommand{\norm}[1]{\Vert #1 \Vert }

\usepackage{anyfontsize}

\usepackage[numbers]{natbib}

\def\tsc#1{\csdef{#1}{\textsc{\lowercase{#1}}\xspace}}
\tsc{WGM}
\tsc{QE}

\begin{document}
\let\WriteBookmarks\relax
\def\floatpagepagefraction{1}
\def\textpagefraction{.001}

\shorttitle{}    

\shortauthors{L.~Banz, M.~Sch\"{o}nauer and A.~Schr\"{o}der }  

\title [mode = title]{Error estimates for perturbed variational inequalities of the first kind}

%

\author[1]{Lothar~Banz}[orcid=0000-0001-9638-9206]
\ead{lothar.banz@plus.ac.at}
\credit{Conceptualization, Methodology, Writing - Original Draft, Writing - Review \& Editing, Software, Validation, Visualization}

\author[1]{Miriam~Sch\"{o}nauer}[orcid=0009-0002-6973-7038]
\ead{miriam.schoenauer@plus.ac.at}
\cormark[1]
\credit{Conceptualization, Methodology, Writing - Review \& Editing}

\author[1]{Andreas~Schr\"{o}der}[orcid=0000-0003-3691-0906]
\ead{andreas.schroeder@plus.ac.at}
\credit{Conceptualization, Methodology, Writing - Original Draft, Writing - Review \& Editing}

\affiliation[1]{organization={Department of Mathematics, University of Salzburg},
            addressline={ Hellbrunner Stra{\ss}e 34}, 
            city={Salzburg},
          citysep={}, 
            postcode={5020}, 
            state={Salzburg},
            country={Austria}}

\cortext[1]{Corresponding author}


\begin{abstract}
    In this paper, we derive a priori error estimates for variational inequalities of the first kind in an abstract framework. This is done by combining the first Strang Lemma and the Falk Theorem. The main application consists in the derivation of a priori error estimates for Galerkin methods, in which "variational crimes" may perturb the underlying variational inequality. Different types of perturbations are incorporated into the abstract framework and discussed by various examples. For instance, the perturbation caused by an inexact quadrature is examined in detail for the Laplacian obstacle problem. For this problem, guaranteed rates for the approximation error resulting from the use of higher-order finite elements are derived. In numerical experiments, the influence of the number of quadrature points on the approximation error and on the quadrature-related error itself is studied for several discretization methods.
\end{abstract}



\begin{keywords}
    variational inequality of the first kind \sep a priori error estimates \sep higher-order finite elements \sep quadrature error \sep variational crimes

    \textit{MSC}
    
    65K15 \sep 65N15 \sep 35J86
\end{keywords}

\maketitle

\section{Introduction}
    Variational inequalities of the first kind are used in the modeling of a variety of problems spanning across multiple fields of research. For instance, obstacle problems, which may be described by such variational inequalities, are of interest in engineering as they can be used to model lubrication, see the monographs \citep{Kinderlehrer2000, RodriguesObstacle}. Additionally, the simulation of fluid displacement in porous media, see \citep{CHANG2017}, is done using variational inequalities. Of course, the use of variational inequalities is not limited to the field of engineering. In logistics, supply chains may also be modeled by variational inequalities, see \citep{TOYASAKI2014,YU2013}, further emphasizing their real-world applicability.
    
    These problems are discretized to compute approximate solutions, usually through the application of finite element methods. Discretizing such variational inequalities, however, often involve "variational crimes" which may introduce perturbations, i.e.~the underlying operator and data are replaced by approximate counterparts. This results in the computation of approximate solutions to the perturbed problem. 
    We will differentiate between two cases: Either the perturbation is introduced
    after the process of discretization (i.e.~formulating the discrete problem is possible but computing the solution is not and thus operator and data are approximated to create a solvable problem) or perturbation is introduced     during the process of discretization (i.e.~operator and data must be approximated in order to produce the formulation of the problem in a finite dimensional space).
    
    One example of instances where perturbation is introduced after the process of discretization is the intentional approximation of the elliptic operator to improve the computational speed. 
    Such techniques are commonly employed in boundary element methods; see, e.g.~\citep{HackbuschBebendorf2003, Hackbusch1999,HackbuschKhoromskij2000Complexity,HackbuschKhoromskij2000, steinbach2008numerical}. The use of an inexact quadrature to compute operator and data also introduce perturbations after the discretization, which is studied in, e.g.~\citep{Ciarlet2002,Rannacher2} in the context of variational equalities. Variational inequalities, constrained by variational equalities, are examples where the perturbation is introduced during discretization. We refer to \citep{BANZ2024219,Banz2022OptimalControl,gwinner1993boundary,Gwinner2018,hinze2008optimization, MAISCHAK2005} for such constrained inequalities.
    
    The aim of this paper is to provide a framework for deriving a priori error estimates which are applicable to a variety of perturbed variational inequality problems. For this, we prove an a priori error estimate (Theorem~\ref{thm:abstractStrangFalk}) through the combination of the first Strang Lemma and the Falk Theorem in an abstract setting. To illustrate the applicability of our result, we discuss examples of the two forms of perturbations as described above.
    
    We give special consideration to the perturbation arising after the discretization process, especially concerning the use of an inexact quadrature. In the framework of the Laplacian obstacle problem, we introduce a higher-order $h$-finite element discretization, analyze the quadrature error, and derive guaranteed convergence rates for a higher-order $h$-finite element discretization setup. To that end, we generalize the ideas used in \citep{Rannacher2} to the variational inequality case.
    
    In the context of perturbations introduced during the process of discretization, we consider variational inequalities constrained by variational equalities. For this purpose, we examine exemplarily the Biot contact problem \citep{BANZ2024219}, BEM for Signorini contact problem \citep{Gwinner2018}, as well as an optimal control problem \citep{Banz2022OptimalControl}.
    
    Finally, in numerical experiments, we illustrate the influence of quadrature rules on the convergence rates of different discretization methods for the obstacle problem. Furthermore, we conduct experiments on the behavior of the quadrature related error itself.
    
    The rest of the article is organized as follows. In Section~\ref{abstractError} we introduce the abstract framework of the variational inequality of the first kind and derive a priori error estimates. In Section~\ref{sec:pert_by_approx} we examine perturbations introduced after discretization, before we study the Laplacian obstacle problem in a discretization setup of higher-order $h$-finite elements and prove a priori convergence rates for the quadrature related error in Section~\ref{sec:perturbed_by_quad}. In Section~\ref{sec:perturbed_by vareq} we review examples of perturbations caused during the process of discretization and discuss several examples.  Finally, in Section~\ref{sec:numer_exp} the numerical experiments are discussed.
    
    Throughout this paper, the term $A\lesssim B$ shall be an abbreviation of $A\leq C\, B$ where $C>0$ is a generic constant. We also use standard notation of Lebesgue and Sobolev spaces with $H^s(\Omega) = W^{s,2}(\Omega)$.

\section{Abstract error estimates}\label{abstractError}

    \label{sec:Appendix}

    Let $V$ be a real Hilbert space and $V^*$ its dual space. Furthermore, let $\widetilde{V} \subset V$ be a second Hilbert space with dual space $\widetilde{V}^*\supset V^*$ where $\|\tilde{v}\|_{\widetilde{V}} = \| \tilde{v} \|_V$ for any $\tilde{v} \in \widetilde{V}$. Let $K$ and $\widetilde{K}$ be two closed, non-empty and convex subsets of $V$ and $\widetilde{V}$, respectively. Furthermore, let $\ell \in V^*$ and $\tilde{\ell}\in \widetilde{V}^*$ be given. Additionally, let $\langle \cdot, \cdot \rangle$ denote the duality pairing between any space and its dual space. We also assume $A:V \rightarrow V^*$ as well  as $\widetilde{A}:\widetilde{V}\rightarrow \widetilde{V}^*$ to be linear, continuous and elliptic mappings, i.e.
    \begin{align*}
        \|Av\|_{V^*} \leq c \|v\|_V, \qquad \langle A v,v\rangle \geq \alpha \|v\|_V^2 \\
        \|\widetilde{A}\tilde{v}\|_{\widetilde{V}^*} \leq \tilde{c} \|\tilde{v}\|_{\widetilde{V}}, \qquad\langle \widetilde{A} \tilde{v},\tilde{v}\rangle \geq \tilde{\alpha} \|\tilde{v}\|_{\widetilde{V}}^2
    \end{align*}
    %
    for all $v\in V$, $\tilde{v} \in \widetilde{V}$, and some constants $c$, $\tilde{c}$, $\alpha$, $\tilde{\alpha}>0$.
    
    We consider the variational inequality problem of finding a $u\in K$ such that
    \begin{align} \label{eq:Varineq}
        \langle A u-\ell ,v-u \rangle \geq 0 \quad  
    \end{align}
    for all $v\in K$.
    
    A perturbation of \eqref{eq:Varineq} is given by the variational inequality: Find a $\tilde{u}\in \widetilde{K}$ such that
    \begin{align} \label{eq:PerturbedVarIneq}  
        \langle \widetilde{A} \tilde{u}-\tilde{\ell},\tilde{v}-\tilde{u} \rangle \geq 0 \quad   
    \end{align}
    for all $\tilde{v} \in \widetilde{K}.$
    
    Given this setup, the two problems \eqref{eq:Varineq} and \eqref{eq:PerturbedVarIneq} are well-defined. Indeed, the variational inequalities \eqref{eq:Varineq} and \eqref{eq:PerturbedVarIneq} have unique solutions, which Lipschitz-continuously depend on the data $\ell$ and $\tilde{\ell}$, respectively, see \cite[Theorem~2.1]{Kinderlehrer2000}. For the derivation of an a priori error estimate, we combine the first Strang Lemma (e.g.~\citep{HackbuschStrang}) with the Falk Theorem (e.g.~\citep{falk1974error}).
    \begin{theorem}\label{thm:abstractStrangFalk}
        There holds
        \begin{align*}
            \| u - \tilde{u} \|_V^2 \leq \left(2+ \frac{4c^2}{\tilde{\alpha}^2} \right)   \| u - \tilde{v} \|_V^2 + \frac{4}{\tilde{\alpha}}\,\langle Au- \ell  , \tilde{v}-u +v -\tilde{u}  \rangle +  \frac{4}{\tilde{\alpha}^2}  \|(A- \widetilde{A}) \tilde{v}-(\ell-\tilde{\ell})\|_{\widetilde{V}^*}^2  
        \end{align*}
        for all $v\in K$ and all $\tilde{v}\in \widetilde{K}$.
    \end{theorem}
    \begin{proof}
        Using \eqref{eq:Varineq} and \eqref{eq:PerturbedVarIneq}, the Cauchy-Schwarz inequality, and Young's inequality, we estimate
        \begin{align*}
            \tilde{\alpha}\, \| \tilde{u} - \tilde{v} \|_V^2 & = \tilde{\alpha}\, \| \tilde{u} - \tilde{v} \|_{\widetilde{V}}^2 \\
            &\leq \langle \widetilde{A} (\tilde{u} - \tilde{v}), \tilde{u} - \tilde{v} \rangle\\
            &\leq \langle \tilde{\ell} , \tilde{u} - \tilde{v} \rangle - \langle \widetilde{A} \tilde{v}, \tilde{u} - \tilde{v} \rangle\\
            &\leq \langle \tilde{\ell} , \tilde{u} - \tilde{v} \rangle + \langle Au- \ell , v-u \rangle   - \langle \widetilde{A} \tilde{v}, \tilde{u} - \tilde{v} \rangle\\
            & = \langle  A\tilde{v} -\tilde{\ell} , \tilde{v} - \tilde{u} \rangle + \langle Au- \ell , v-u \rangle  - \langle A\tilde{v} , \tilde{v}-\tilde{u} \rangle   - \langle \widetilde{A} \tilde{v}, \tilde{u} - \tilde{v} \rangle\\
            & = \langle Au- \ell  , \tilde{v} -\tilde{u}  \rangle + \langle Au- \ell , v-u \rangle + \langle A(\tilde{v}-u) , \tilde{v}-\tilde{u} \rangle  +  \langle (A- \widetilde{A}) \tilde{v}, \tilde{u} - \tilde{v} \rangle + \langle \ell-\tilde{\ell}  , \tilde{v} -\tilde{u}  \rangle\\
            & \leq \langle Au- \ell  , \tilde{v}-u +v -\tilde{u}  \rangle  + c\,\| u-\tilde{v}\|_V \,\|\tilde{u}-\tilde{v} \|_V  +  \|(A- \widetilde{A}) \tilde{v}-(\ell-\tilde{\ell})\|_{\widetilde{V}^*}\, \| \tilde{u} - \tilde{v} \|_{\widetilde{V}} \\
            & \leq \langle Au- \ell  , \tilde{v}-u +v -\tilde{u}  \rangle  + \frac{1}{2\tilde{\alpha}}\,\left( c\,\| u-\tilde{v}\|_V   +   \|(A- \widetilde{A}) \tilde{v}-(\ell-\tilde{\ell})\|_{\widetilde{V}^*}  \right)^2 + \frac{\tilde{\alpha}}{2}\, \|\tilde{u}-\tilde{v} \|_V^2.
        \end{align*}
        This gives
        \begin{align*}
            \| \tilde{u} - \tilde{v} \|_V^2  \leq \frac{2}{\tilde{\alpha}}\, \langle Au- \ell  , \tilde{v}-u +v -\tilde{u}  \rangle  + \frac{2c^2}{\tilde{\alpha}^2}\,\| u-\tilde{v}\|_V^2   +  \frac{2}{\tilde{\alpha}^2}\, \|(A- \widetilde{A}) \tilde{v} -(\ell-\tilde{\ell})\|_{\widetilde{V}^*}^2.
        \end{align*}
        Hence, exploiting $\| u - \tilde{u} \|_V^2 \leq 2\,\| u - \tilde{v} \|_V^2 + 2\,\|  \tilde{u} - \tilde{v} \|_V^2$ completes the proof.
    \end{proof}
    \begin{remark}
        Clearly, Theorem~\ref{thm:abstractStrangFalk} contains the pure Falk Theorem (i.e.~$\widetilde{A}=A$, $\tilde{\ell}=\ell$ and $\widetilde{V} = V$) and the pure Strang Lemma (i.e.~$K$ and $\widetilde{K}$ are vector spaces thus $Au-\ell = 0$), but both with slightly larger error constants than the original results. 
    \end{remark}
    \begin{remark}\label{rem:uniqueness}
        At this stage, Theorem~\ref{thm:abstractStrangFalk} admits an alternative proof for the uniqueness and the Lipschitz dependence on the data $\ell$ and $\tilde{\ell}$.
        
        The uniqueness follows directly from Theorem~\ref{thm:abstractStrangFalk} by setting $\widetilde{V}=V$, $\widetilde{K}=K$, $\widetilde{A}=A$, $\tilde{\ell}=\ell$, $\tilde{v}=u$ and $v=\tilde{u}$, where $\tilde{u}$ is assumed to be an additional solution to \eqref{eq:Varineq}.
        
        To prove the Lipschitz dependency, assume $u$, $\tilde{u}$ is the solution of \eqref{eq:Varineq} to the data $\ell$, $\tilde{\ell}$, respectively. Then, by setting $\widetilde{V}=V$, $\widetilde{K}=K$, $\widetilde{A}=A$, $\tilde{v}=u$ and $v=\tilde{u}$, Theorem~\ref{thm:abstractStrangFalk} yields with $\tilde{\alpha} = \alpha$ that
        \begin{align*}
            \| u - \tilde{u} \|_V \leq  \frac{2}{\alpha}  \|\ell-\tilde{\ell}\|_{V^*}
        \end{align*}
        and thus recovers the Lipschitz-continuous dependency on the data $\ell$, but with a larger Lipschitz-constant (by a factor of 2) than obtained by a direct proof resulting from the Strang Lemma part in the proof of Theorem~\ref{abstractError}.
    \end{remark}
    The main application of Theorem~\ref{thm:abstractStrangFalk} lies in the a priori error estimate for Galerkin methods, in which the convex set $K$ is discretized to some $\widetilde{K}$ but also the operator $A$ and the right-hand side data $\ell$ themselves are discretized to some $\widetilde{A}$ and $\tilde{\ell}$.

\section{Perturbation introduced after discretization} \label{sec:pert_by_approx}
    Discretizing the variational inequality problem of \eqref{eq:Varineq}
    typically leads to the discrete problem of finding a $u^* \in \widetilde{K} \subset \widetilde{V}$ such that
    \begin{align} \label{eq:vi_disc}
        \langle A u^* - \ell, \tilde{v}-u^* \rangle \geq 0 
    \end{align}
    for all $\tilde{v} \in \widetilde{K}$. In most cases, computing the solution $u^*$ is impossible. Instead, the operator $A$ and data $\ell$ are replaced by approximative counterparts $\widetilde{A}$ and $\tilde{\ell}$ in order to produce a solvable problem. Thus, we end up solving only the perturbed (discrete) problem \eqref{eq:PerturbedVarIneq}.
    
    This kind of perturbation may result from an inexact quadrature or desired approximations to speed up the computation. The latter frequently occurs, for instance, in boundary element methods with their non-local operators, and we refer to the book \citep{steinbach2008numerical} and the references therein for an introduction to fast BEM by approximating the operator $A$.
    
    Applying Theorem~\ref{thm:abstractStrangFalk} gives us an easy way to bound the total approximation error $u-\tilde{u}$. If we apply Theorem~\ref{thm:abstractStrangFalk} slightly differently, however, we can show that the constants in front of the first two terms on the right-hand side in Theorem~\ref{thm:abstractStrangFalk} are indeed independent of $\tilde{\alpha}$ and thus independent of the approximation $\widetilde{A} \approx A$.
    \begin{corollary}\label{thm:AprioriErrorPertByApprox}
        There holds
        \begin{align*}
            \| u - \tilde{u} \|_V^2 \leq \left(4+ \frac{8c^2}{\alpha^2} \right)   \| u - \tilde{v} \|_V^2 + \frac{8}{\alpha}\,\langle Au- \ell  , \tilde{v}-u +v -u^*  \rangle  + \frac{8}{\tilde{\alpha}^2}\,  \|(A- \widetilde{A}) u^*-(\ell-\tilde{\ell})\|_{\widetilde{V}^*}^2
        \end{align*}
        for all $v\in K$ and all $\tilde{v} \in \widetilde{K}$.
    \end{corollary}
    \begin{proof}
        We apply Theorem~\ref{thm:abstractStrangFalk} to \eqref{eq:Varineq} and \eqref{eq:vi_disc} and obtain the classical result by Falk \citep{falk1974error} (up to some enlarged constants) as no approximation of the operator and the right hand side data occurs, i.e.
        \begin{align*}
            \| u - u^* \|_V^2 \leq \left(2+ \frac{4c^2}{\alpha^2} \right)   \| u - \tilde{v} \|_V^2 + \frac{4}{\alpha}\,\langle Au- \ell  , \tilde{v}-u +v -u^*  \rangle 
        \end{align*}
        for all $v\in K$ and all $\tilde{v} \in \widetilde{K}$.
        
        Next, we apply Theorem~\ref{thm:abstractStrangFalk} to \eqref{eq:vi_disc} and \eqref{eq:PerturbedVarIneq} with $\tilde{v} = u^*$ and $v = \tilde{u}$ as no change in the convex set $\widetilde{K}$ occurs. Thus, we obtain 
        \begin{align*}
            \| u^* - \tilde{u} \|_V^2 =   \| u^* - \tilde{u} \|_{\widetilde{V}}^2 \leq  \frac{4}{\tilde{\alpha}^2}\,  \|(A- \widetilde{A}) u^*-(\ell-\tilde{\ell})\|_{\widetilde{V}^*}^2 .
        \end{align*}
        Application of the triangle inequality, i.e.~$\| u - \tilde{u} \|_V^2 \leq 2\| u - u^* \|_V^2 + 2\| u^* - \tilde{u}\|_V^2$, yields the assertion.
    \end{proof}
    We emphasize that the first two terms $\| u - \tilde{v} \|_V^2 $ and $\langle Au- \ell  , \tilde{v}-u +v -u^*  \rangle $ in Theorem~\ref{thm:AprioriErrorPertByApprox} form the classical a priori error estimate (see e.g.~\citep{RodriguesObstacle}), and bounds of these can be found in the literature for the specific problem at hand (see e.g.~\citep{OdenKikuchiPorousMedia}).
    
    The third and last term $\|(A- \widetilde{A}) u^*-(\ell-\tilde{\ell})\|_{\widetilde{V}^*}^2$ is the approximation error of the operator applied to some (uniformly) bounded function $u_N$ and of the right-hand side. Both approximations in that term are independent of the underlying variational problem, and form a literature strand of its own. 

\subsection{Quadrature error and the Laplacian obstacle problem}\label{sec:perturbed_by_quad}
    In this section, we analyze the perturbation (error) emerging from an inexact quadrature, which is used in the context of a Laplacian obstacle problem. Let $\Omega \subset \mathbb{R}^d$, $d\in \{2,3\}$, be a bounded, polygonal Lipschitz domain. We seek a function $u$ that solves
    \begin{align*}
        -\operatorname{div} (a\,\nabla u) \geq f, \  u \geq \psi, \  (f+\operatorname{div} (a\,\nabla u))\,(u-\psi)&=0 \text{ in } \Omega\\ u &= 0 \text{ on } \partial \Omega
    \end{align*}
    in a weak sense, where $a\in W^{r,\infty}(\Omega)$ with $a\geq a_0 >0$ and $r \geq d$, $f \in L^2(\Omega)$ and $\psi \in C^0(\Omega)$ with $\psi|_{\partial \Omega} \leq 0$. Its standard variational inequality formulation is to find a $u \in K = \left\{ v \in H^1_0(\Omega) \mid v \geq \psi \right\}$ such that
    \begin{align} \label{eq:obstacleVI}
        \int_\Omega a\,\nabla u\, \nabla (v-u) \, dx \geq \int_\Omega f (v-u) \, dx
    \end{align}
    for all $v \in K.$
    Obviously, we obtain \eqref{eq:obstacleVI} from \eqref{eq:Varineq} by setting 
    \begin{align*}
        V :=H^1_0(\Omega),\qquad
        V^*:= H^{-1}(\Omega), \qquad 
        K := \left\{ v \in H^1_0(\Omega) \mid v \geq \psi \right\},
    \end{align*}
    and
    \begin{align*}
        \langle Av,w \rangle  := (a\,\nabla v, \nabla w)_{L^2(\Omega)}, \qquad\langle
        \ell,v\rangle:=(f,v)_{L^2(\Omega)}
    \end{align*}
    for any $v,w\in H^1_0(\Omega)$ and where $(\cdot,\cdot)_{L^2(\Omega)}$ denotes the standard $L^2(\Omega)$ inner product. It is well known that $A:H^1_0(\Omega)\rightarrow H^{-1}(\Omega)$ is continuous and $H^1_0(\Omega)$-elliptic.
    %
    %
    \subsection{\texorpdfstring{A higher-order $h$-finite element discretization}{A higher-order $h$-finite element discretization}} \label{sec:hp_discretization}
        Let $\mathcal{D}_h$ be a locally quasi-uniform and $\gamma$-shape regular decomposition of $\Omega$ into triangles or parallelograms if $d=2$ or tetrahedrons or parallelepipeds if $d=3$. For an element $D\in\mathcal{D}_h$, the local element size is given by $h_D=\operatorname{diam}(D)$ and the global mesh size by $h=\max_{D\in\mathcal{D}_h} h_D$. Furthermore, let $p\geq 1$ be the uniform polynomial degree on $\mathcal{D}_h$. We refer to \cite[Section~1.2]{melenk2001residual} for a definition of locally quasi-uniform and $\gamma$-shape regular decompositions.
        
        In particular, we have an affine and bijective mapping $F_D:\hat{D} \rightarrow D$ from the unit triangle, the unit square, the unit tetrahedron, or the unit cube $\hat{D}$, respectively, onto an element $D$ of $\mathcal{D}_h$. 
        We point out that the $\gamma$-shape regularity of $\mathcal{D}_h$ means
        \begin{equation*}
            h_D^{-1} \|\nabla F_D\|+h_D\|(\nabla F_D)^{-1}\|\leq \gamma
        \end{equation*}
        for some constant $\gamma>0$ and Euclidean norm $\|\cdot\|$, and implies
        \begin{equation}\label{nablaFT}
            \gamma^{-1} \|\nabla (f\circ F_D)(\hat{x})\|\leq h_D \|\nabla f\circ F_D(\hat{x})\|\leq \gamma\, \|\nabla (f\circ F_D)(\hat{x})\|
        \end{equation}
        for any $f\in H^1(D)$ and $\hat{x}\in \hat{D}$.
        Using the polynomial spaces 
        \begin{equation*}
            \mathbb{P}_{p}(D):=\operatorname{span} \Big\{ \prod_{i=1}^d x_i^{j_i} \:| \: j\in \mathcal{I}(D)\Big\}
        \end{equation*}
        with
        \begin{equation*}
            \mathcal{I}(D):= \left\{j\in\mathbb{N}^d \mid
            \begin{cases}
                \sum_{i=1}^d j_i\leq p ,& D\text{ is a triangle or tetrahedron},\\
                j_1,\ldots,j_d\leq p, & D\text{ is a parallelogram or parallelepiped}
            \end{cases}\right\}
        \end{equation*}
        we define the discretization spaces 
        \begin{align*}
            W_{hp}&:= \left\{ v_{hp} \in L^2(\Omega) : v_{hp}|_D \circ F_D \in \mathbb{P}_{p}(\hat{D})\ \:\forall D \in \mathcal{D}_h \right\},\\
            V_{hp}&:=W_{hp}\cap H^1_0(\Omega).
        \end{align*}
        Let $\{\hat{\zeta}_i\}_{i=0}^{n_p}\subset \hat{D}$ for some $n_p \in \mathbb{N}_0$ and set
        \begin{align*}
            G_{hp} := \bigcup_{D\in \mathcal{D}_h} \bigcup_{i=0}^{n_p} F_D(\hat{\zeta}_i).
        \end{align*}
        Therewith, we define the discrete, non-empty, closed and convex set 
        \begin{align*}
            K_{hp} := \left\{ v_{hp} \in V_{hp} \mid v_{hp}(\zeta) \geq \psi(\zeta) \: \: \forall \zeta \in G_{hp} \right\}.
        \end{align*}
        The discrete variational inequality problem becomes: Find a $u_{hp} \in K_{hp}$ such that
        \begin{align*} 
            \int_\Omega a\,\nabla u_{hp} \nabla (v_{hp}-u_{hp}) \, dx \geq \int_\Omega f  (v_{hp}-u_{hp}) \, dx 
        \end{align*}
        for all $v_{hp} \in K_{hp}.$    
    %
    %
    \subsection{Perturbations by quadrature}\label{sec: PerturbationQuadrature}
        Let $\{\hat{\xi}_i\}_{i=0}^{m_p}\subset \hat{D}$ for some $m_p\in\mathbb{N}_0$ with $m_p\geq \dim \mathbb{P}_p(\hat{D})$ be the quadrature points and $\{\omega_i\}_{i=0}^{m_p} \subset \mathbb{R}^+$ their corresponding positive quadrature weights. Furthermore, we set $\{\xi_i^D\}_{i=0}^{m_p}:=\{F_D(\hat{\xi}_i)\}_{i=0}^{m_p}$ and $\{\omega_i^D\}_{i=0}^{m_p}:=\{|\det \nabla F_D|\omega_i\}_{i=0}^{m_p}$ for $D\in\mathcal{D}_h$. The quadrature formulas are thus
        \begin{align*}
            Q_{\hat{D}}(\hat{f})&:=\sum_{i=0}^{m_p} \omega_i  \hat{f}(\hat{\xi}_i),\quad  Q_D (f) := \sum\limits_{i=0}^{m_p}\omega_i^D f(\xi_i^D) = |\det \nabla F_D|Q_{\hat{D}}(f\circ F_D),\\
            Q_{hp}(\tilde{f})&:=\sum_{D\in \mathcal{D}_h} Q_D(\tilde{f}),  
        \end{align*}
        where $\hat{f}\in H^{1+t}(\hat{D})$, $f\in H^{1+t}(D)$ and $\tilde{f}\in H^{1+t}(\Omega)$ with $t > d/2-1 \geq 0$, as the functions have to be evaluated in specific points.
        We assume that the quadrature formulas $Q_{\hat{D}}$ and $Q_D$ are of order $r\in\mathbb{N}$, i.e.~polynomials up to degree $r-1$ are integrated exactly, and that for all $v\in \mathbb{P}_p(D)$ it holds that 
        \begin{align}\label{eq: admissable quadrature}
            \forall\: 0\leq i\leq m_p: \nabla v(\xi_i^D)=0\quad  \Rightarrow\quad v \equiv const.
        \end{align}
        Furthermore, we note that $|\det \nabla F_D|$ is constant as $F_D$ is affine as well as
        \begin{equation}\label{IntegralFT}
            \int_D f\,dx = |\det \nabla F_D| \int_{\hat{D}} f\circ F_D\,d\hat{x}
        \end{equation}
        for any $f\in L^2(D)$.
        Next, we consider the perturbed problem of finding a $\tilde{u}_{hp} \in K_{hp}$ such that
        \begin{align*} 
            Q_{hp}\left(a\,\nabla \tilde{u}_{hp}^\top \nabla (v_{hp}-\tilde{u}_{hp}) \right) \geq Q_{hp}\left( f \, (v_{hp}-\tilde{u}_{hp}) \right) 
        \end{align*}
        for all $v_{hp} \in K_{hp}$, provided that $f$ is continuous.
        We define $\widetilde{A}:V_{hp}\rightarrow V_{hp}^*$ and $\tilde{\ell}\in V_{hp}^*$ as
        \begin{equation*}
            (\widetilde{A}v_{hp})(z_{hp}):= Q_{hp}(a\,\nabla v_{hp}^\top \nabla z_{hp}),\qquad 
            \tilde{\ell}(z_{hp}):= Q_{hp}(fz_{hp})
        \end{equation*}
        for any $v_{hp},z_{hp}\in V_{hp}$. Before we can prove the continuity and ellipticity of $\widetilde{A}$, we need to establish the following three auxiliary results.
        \begin{lemma}\label{QCauchy}
            It holds
            \begin{equation*}
                Q_{\hat{D}}(\hat{f}^\top\hat{g}) \leq \bigg[Q_{\hat{D}}(\|\hat{f}\|^2)\bigg]^{1/2}\, \bigg[Q_{\hat{D}}(\|\hat{g}\|^2)\bigg]^{1/2}
            \end{equation*}
            for all $\hat{f},\hat{g}\in [H^{1+t}(\hat{D})]^k$ and some $k\in \mathbb{N}$ with $t > d/2-1 \geq 0$.
        \end{lemma}
        \begin{proof}
            From the Cauchy-Schwarz inequality and $\omega_i>0$ we obtain
            \begin{align*}
                Q_{\hat{D}}(\hat{f}^\top \hat{g}) & =\sum_{i=0}^{m_p} (\omega_i^{1/2}\; \hat{f}(\hat{\xi_i}))^\top\: (\omega_i^{1/2}\; \hat{g}(\hat{\xi_i}))\\
                &\leq \left(\sum_{i=0}^{m_p} \sum_{j=1}^k \omega_i\, \hat{f}_j^2(\hat{\xi_i})\right)^{1/2}\cdot \left(\sum_{i=0}^{m_p}\sum_{j=1}^k\omega_i\, \hat{g}_j^2(\hat{\xi_i})\right)^{1/2}\\
                & = \left(\sum_{i=0}^{m_p} \omega_i\|\hat{f}(\hat{\xi_i})\|^2\right)^{1/2}\cdot \left(\sum_{i=0}^{m_p}\omega_i \|\hat{g}(\hat{\xi_i})\|^2\right)^{1/2}. 
            \end{align*}
        \end{proof}
        \begin{lemma}\label{Qfinite}
            There exist constants $c_p,d_p>0$ depending on $p$ such that
            \begin{align*}
                c_p^{-1}\,\|\nabla \hat{f}\|_{L^2(\hat{D})}^2 \leq\, &Q_{\hat{D}}(\|\nabla \hat{f}\|^2) \leq c_p\, \|\nabla \hat{f}\|_{L^2(\hat{D})}^2,\\
                d_p^{-1}\,\|\hat{f}\|_{L^2(\hat{D})}^2 \leq\, &Q_{\hat{D}}(\hat{f}^{\,2}) \leq d_p\, \|\hat{f}\|_{L^2(\hat{D})}^2\quad
            \end{align*}
            for all $\hat{f}\in \mathbb{P}_p(\hat{D})$.
        \end{lemma}
        \begin{proof}
            Using $\omega_i>0$, $m_p\geq \dim \mathbb{P}_p(\hat{D})$ and Lemma~\ref{QCauchy} it is easy to verify that $\hat{f}\mapsto (Q_{\hat{D}}(\|\nabla \hat{f}\|^2)^{1/2}$ and $\hat{f}\mapsto (Q_{\hat{D}}(\hat{f}^{\,2}))^{1/2}$ are norms in $\mathbb{P}_p(\hat{D})/\mathbb{R}$ and $\mathbb{P}_p(\hat{D})$, respectively. We note that \eqref{eq: admissable quadrature} is used to guarantee the positive definiteness of the norm in $\mathbb{P}_p(\hat{D})/\mathbb{R}$. The assertion follows from the norm equivalence of finite dimensional spaces.
        \end{proof}
        \begin{lemma}\label{QLem}
            There exist constants $\hat{c},\hat{d}>0$ depending on $p$ such that
            \begin{equation*}
                Q_{hp}(a\,\nabla v_{hp}^\top \nabla z_{hp})\leq \hat{c}\, |v_{hp}|_{H^1(\Omega)}\, |z_{hp}|_{H^1(\Omega)}\: \text{ and }\quad Q_{hp}(f z_{hp})\leq \hat{d}\,  \|z_{hp}\|_{L^2(\Omega)}  
            \end{equation*}
            for all $v_{hp},z_{hp}\in V_{hp}$.
        \end{lemma}
        \begin{proof}
            From \eqref{nablaFT}, \eqref{IntegralFT}, Lemma~\ref{QCauchy}, Lemma~\ref{Qfinite} and Cauchy Schwarz inequality, we conclude
            \begin{align*}
                Q_{hp}(a\,\nabla u_{hp}^\top \nabla z_{hp}) & =\sum_{D\in \mathcal{D}_h} |\det \nabla F_D|\, Q_{\hat{D}}\bigg((a\circ F_D)\, (\nabla u_{hp}\circ F_D)^\top (\nabla z_{hp}\circ F_D)\bigg)\\
                & \leq \|a\|_{L^\infty(\Omega)} \sum_{D\in \mathcal{D}_h} |\det \nabla F_D| \,
                \bigg[Q_{\hat{D}}(\|\nabla u_{hp}\circ F_D\|^2)\bigg]^{1/2}\, \bigg[Q_{\hat{D}}(\|\nabla z_{hp}\circ F_D\|^2)\bigg]^{1/2}\\
                & \leq \gamma^2\, \|a\|_{L^\infty(\Omega)} \sum_{D\in \mathcal{D}_h} h_D^{-2} |\det \nabla F_D| \,
                \bigg[Q_{\hat{D}}(\|\nabla (u_{hp}\circ F_D)\|^2)\bigg]^{1/2}\, \bigg[Q_{\hat{D}}(\|\nabla (z_{hp}\circ F_D)\|^2)\bigg]^{1/2}\\
                & \leq c_p\,\gamma^2\, \|a\|_{L^\infty(\Omega)} \sum_{D\in \mathcal{D}_h} |\det \nabla F_D|\, h_D^{-2} \|\nabla (u_{hp}\circ F_D)\|_{L^2(\hat{D})}\, \|\nabla (z_{hp}\circ F_D)\|_{L^2(\hat{D})}\\
                & \leq c_p\, \gamma^4\, \|a\|_{L^\infty(\Omega)} \sum_{D\in \mathcal{D}_h} |\det \nabla F_D|\, \|\nabla u_{hp}\circ F_D\|_{L^2(\hat{D})}\, \|\nabla z_{hp}\circ F_D\|_{L^2(\hat{D})}\\
                & =  c_p\, \gamma^4\, \|a\|_{L^\infty(\Omega)} \sum_{D\in \mathcal{D}_h} \|\nabla u_{hp}\|_{L^2(D)}\, \|\nabla z_{hp}\|_{L^2(D)}\\
                & \leq   c_p\, \gamma^4\, \|a\|_{L^\infty(\Omega)}\, \|\nabla u_{hp}\|_{L^2(\Omega)}\, \|\nabla z_{hp}\|_{L^2(\Omega)}
            \end{align*}
            and
            \begin{align*}
             Q_{hp}(f z_{hp}) & =\sum_{D\in \mathcal{D}_h} |\det \nabla F_D|\, Q_{\hat{D}}\bigg((f\circ F_D)\,(z_{hp}\circ F_D)\bigg)\\
            & \leq d_p \sum_{D\in \mathcal{D}_h} |\det \nabla F_D|\, \|f\circ F_D\|_{L^2(\hat{D})} \,\|z_{hp}\circ F_D\|_{L^2(\hat{D})}\\
            & = d_p \sum_{D\in \mathcal{D}_h} \|f\|_{L^2(D)} \,\|z_{hp}\|_{L^2(D)}\\
            & \leq d_p \, \|f\|_{L^2(\Omega)} \,\|z_{hp}\|_{L^2(\Omega)}.
            \end{align*}
        \end{proof}
        With these lemmas established, we can prove the continuity and the ellipticity of the perturbed operator $\widetilde{A}$.
        \begin{theorem}\label{pertTh}
            There exist constants $\tilde{c},\tilde{\alpha}>0$ depending on $p$ such that
            \begin{equation*}
                \|\widetilde{A} z_{hp}\|_{V_{hp}^*}\leq \tilde{c}\,\|z_{hp}\|_{H^1(\Omega)} \:\text{ and }\: \langle\widetilde{A}z_{hp},z_{hp}\rangle\geq  \tilde{\alpha}\, \|z_{hp}\|_{H^1(\Omega)}^2 
            \end{equation*}
            for all $z_{hp}\in V_{hp}$.
        \end{theorem}
        \begin{proof}
            Lemma~\ref{QLem} implies
            \begin{align*}
                \|\widetilde{A} z_{hp}\|_{V_{hp}^*} & = \sup_{\substack{w_{hp}\in V_{hp},\\ \|w_{hp}\|_{H^1(\Omega)}=\,1}} \langle \widetilde{A}z_{hp},w_{hp}\rangle
                =\sup_{\substack{w_{hp}\in V_{hp},\\ \|w_{hp}\|_{H^1(\Omega)}=\,1}}  Q_{hp}(a\, \nabla z_{hp}^\top \nabla w_{hp})
                 \leq \hat{c}\, \|z_{hp}\|_{H^1(\Omega)}.  
            \end{align*}
            Applying Lemma~\ref{Qfinite} we conclude
            \begin{align*}
                \langle\widetilde{A}z_{hp},z_{hp}\rangle & = Q_{hp}(a\,\|\nabla z_{hp}\|^2)\\
                &=\sum_{D\in \mathcal{D}_h} |\det \nabla F_D|\, Q_{\hat{D}}(a\circ F_D\,\|\nabla z_{hp}\circ F_D\|^2)\\
                & \geq d_p^{-1} \sum_{D\in \mathcal{D}_h} |\det \nabla F_D|\, \|(a\circ F_D)^{1/2}\nabla z_{hp}\circ F_D\|^2_{L^2(\hat{D})}\\
                & \geq d_p^{-1} a_0 \sum_{D\in \mathcal{D}_h} |\det \nabla F_D|\, \|\nabla z_{hp}\circ F_D\|^2_{L^2(\hat{D})}\\
                & \geq d_p^{-1}a_0 \,\|\nabla z_{hp}\|_{L^2(\Omega)}^2.
            \end{align*} 
        \end{proof}
        In preparation for the main result of this section, we will prove the following lemma. The proofs of this auxiliary result and the main result use similar ideas as proposed in \citep{Rannacher2}. 
        \begin{lemma}\label{lem:Quad_estimate}
            For the quadrature formula $Q_D(\cdot)$ of order $r\in\mathbb{N}$ with $r\geq d$ it holds that 
            \begin{align}\label{eq:Quadrature_estimate}
                \left| \int_D w\,v \, dx -Q_D(w\,v) \right| \lesssim  h_D^r\, \norm{w}_{H^r(D)}\norm{v}_{H^r(D)}
            \end{align}
            for $w,v\in H^r(D)$ and $D\in \mathcal{D}_h$.
        \end{lemma}
        \begin{proof}
            Applying  \cite[Theorem 3.15]{Rannacher2} it holds that 
            \begin{align}\label{eq:quad_formel_basis}
               \left| \int_D w\,v \, dx -Q_D(w\,v) \right| \lesssim  h_D^r |w\,v|_{W^{r,1}(D)}.
            \end{align}
            Then by the Leibniz formula, see e.g.~\cite[Chapter 5.2]{Evans2002}, as well as the Hölder inequality we obtain
            \begin{align}\label{eq:ProductruleArg}
                |w\,v|_{W^{r,1}(D)} &= \sum\limits_{|\alpha|= r}\norm{D^\alpha(w\,v)}_{L^1(D)}\\
                &=\sum\limits_{|\alpha|= r} \Big\Vert\sum\limits_{|\beta|\leq |\alpha|}\binom{\alpha}{\beta}D^{\beta}w\,D^{\alpha-\beta}v\Big\Vert_{L^1(D)}\nonumber\\
                &\lesssim \sum\limits_{|\alpha|= r} \sum\limits_{|\beta|\leq r}\norm{D^{\beta}w\,D^{\alpha-\beta}v}_{L^1(D)}\nonumber\\
                &\lesssim \sum\limits_{|\alpha|= r} \sum\limits_{|\beta|\leq r} \norm{D^{\beta}w}_{L^2(D)}\norm{D^{\alpha-\beta}v}_{L^2(D)}\nonumber\\
                &\lesssim \sum\limits_{|\beta|\leq r} \norm{D^{\beta}w}_{L^2(D)}\,\sum\limits_{|\alpha|= r} \sum\limits_{|\beta|\leq r}\norm{D^{\alpha-\beta}v}_{L^2(D)}\nonumber\\ 
                &\lesssim \Big(\sum\limits_{|\beta|\leq r} \norm{D^{\beta}w}_{L^2(D)}^2\Big)^{1/2}\,\Big(\sum\limits_{|\gamma|\leq r}\norm{D^{\gamma}v}_{L^2(D)}^2\Big)^{1/2}\nonumber\\
                &\lesssim \norm{w}_{H^r(D)}\,\norm{v}_{H^r(D)}.
            \end{align}
            The estimate \eqref{eq:quad_formel_basis} combined with \eqref{eq:ProductruleArg} yields the desired inequality \eqref{eq:Quadrature_estimate}.
        \end{proof}
        \begin{theorem}\label{Quad Error}
            Let $Q_D(\cdot)$ be of order $r\in\mathbb{N}$ with $r \geq \max\{p,d\}$. Furthermore, let $f \in H^{r}(\Omega)$, $a \in W^{r,\infty}(\Omega)$, $u \in H^s(\Omega)$ for some $s\geq 1$ and $u_{hp}$ satisfies the a priori error estimate
            \begin{align} \label{eq:assumed_error_estimate_uhp}
                \| u - u_{hp}\|_{H^1(\Omega)} \lesssim  h^\beta
            \end{align}
            for some $\beta >0 $. Then there holds
            \begin{align*}
                \|(A- \widetilde{A}) u_{hp}-(\ell-\tilde{\ell})\|_{V_{hp}^*} & \lesssim h^{r+1-p} \left( h^{\beta + 1-p} + h^{\min\{0,\lfloor s\rfloor-p\}} \right).
            \end{align*}
        \end{theorem}
        \begin{proof}
            Using \eqref{eq:Quadrature_estimate} we estimate
            \begin{align*}
                \left|\langle \ell - \tilde{\ell},v_{hp} \rangle \right| &= \left| \sum_{D \in  \mathcal{D}_h}  \int_{D} f\, v_{hp}\, dx -Q_D(f\,v_{hp}) \right| \lesssim \sum_{D \in  \mathcal{D}_h} h_{D}^r\, \|f\|_{H^r(D)}\, \|v_{hp}\|_{H^r(D)}.
            \end{align*}
            Recalling that $v_{hp}|_D$ is a polynomial of degree $p$ as $F_D$ is affine, and $r \geq p$ we obtain
            \begin{align*}
                \left|\langle \ell - \tilde{\ell},v_{hp} \rangle \right|  & \lesssim  \sum_{D \in  \mathcal{D}_h} h_D^r \,\|f\|_{H^r(D)}\, \|v_{hp}\|_{H^p(D)}.
            \end{align*}
            Then, by employing an inverse estimate for polynomials
            \begin{align}\label{eq:PolynomialsInverse}
                \norm{v_{hp}}_{H^k(D)}\lesssim h^{j-k}\norm{v_{hp}}_{H^j(D)},
            \end{align}
            for all $0\leq j\leq k\leq p+1$ and all $v_{hp}\in V_{hp}$, see e.g.~\cite[Theorem 6.8]{Braess_2007}, as well as the Cauchy-Schwarz inequality we arrive at the estimate
            \begin{align}
                \left|\langle \ell - \tilde{\ell},v_{hp} \rangle \right| & \lesssim  \sum_{D \in  \mathcal{D}_h} h_D^{r+1-p}\, \|f\|_{H^r(D)} \,\|v_{hp}\|_{H^1(D)}\nonumber  \\
                & \lesssim  \left(\sum_{D \in  \mathcal{D}_h} h_D^{2(r+1-p)}\, \|f\|_{H^r(D)}^2 \right)^{1/2} \left(\sum_{D \in  \mathcal{D}_h} \|v_{hp}\|_{H^1(D)}^2\right)^{1/2}\nonumber\\
                & = \left(\sum_{D \in  \mathcal{D}_h} h_D^{2(r+1-p)}\, \|f\|_{H^r(D)}^2 \right)^{1/2} \|v_{hp}\|_{H^1(\Omega)} . \label{eq:quad_error_rhs}
            \end{align}
            For the quadrature error with respect to the operator $A$, we proceed similarly. With  \eqref{eq:quad_formel_basis} and the Leibniz formula, we get the estimate 
            \begin{align*}
                | \langle (A- \widetilde{A}) u_{hp},v_{hp}\rangle | &= \left| \sum_{D \in  \mathcal{D}_h} \int_{D} a\, \nabla u_{hp} \nabla v_{hp} \, dx - Q_D(a\nabla u_{hp} \nabla v_{hp})\right| \nonumber \\
                & \lesssim   \sum_{D \in  \mathcal{D}_h} h_{D}^r \left|a \,\nabla u_{hp} \nabla v_{hp} \right|_{W^{r,1}(D)}  \nonumber  \\
                & \lesssim   \sum_{D \in  \mathcal{D}_h} h_D^r \sum\limits_{|\alpha|=r}\norm{D^{\alpha}\big(a\, \nabla u_{hp} \nabla v_{hp}\big)}_{L^1(D)}\\
                & \lesssim   \sum_{D \in  \mathcal{D}_h} h_D^r \sum\limits_{|\alpha|=r} \sum\limits_{|\beta|\leq r}\Vert D^{\beta}\big(a\, \nabla u_{hp}\big)D^{\alpha-\beta}(\nabla v_{hp})\Vert_{L^1(D)}\\
                & \lesssim   \sum_{D \in  \mathcal{D}_h} h_D^r \sum\limits_{|\alpha|=r} \sum\limits_{|\beta|\leq r}\sum\limits_{|\gamma|\leq|\beta|}\norm{D^{\gamma}(a)\,D^{\beta-\gamma}(\nabla u_{hp})D^{\alpha-\beta}(\nabla v_{hp})}_{L^1(D)}.
            \end{align*}
            Applying Hölder's inequality yields
            \begin{align*}
                | \langle (A- \widetilde{A}) u_{hp},v_{hp}\rangle | &\lesssim  \sum_{D \in  \mathcal{D}_h} h_D^r \sum\limits_{|\alpha|=r} \sum\limits_{|\beta|\leq r}\sum\limits_{|\gamma|\leq|\beta|} \norm{D^{\gamma}(a)}_{L^{\infty}(D)}\norm{D^{\beta-\gamma}(\nabla u_{hp})D^{\alpha-\beta}(\nabla v_{hp})}_{L^1(D)}\\
                & \lesssim  \sum_{D \in  \mathcal{D}_h} h_D^r \norm{a}_{W^{r,\infty}(D)}\sum\limits_{|\alpha|=r} \sum\limits_{|\beta|\leq r}\sum\limits_{|\gamma|\leq|\beta|}\norm{D^{\beta-\gamma}(\nabla u_{hp})}_{L^2(D)}\norm{D^{\alpha-\beta}(\nabla v_{hp})}_{L^2(D)}.
            \end{align*}
            Then, with the same arguments as in the proof of Lemma~\ref{lem:Quad_estimate} we can conclude
            \begin{align*}
                | \langle (A- \widetilde{A}) u_{hp},v_{hp}\rangle | & \lesssim  \|a\|_{W^{r,\infty}(\Omega)} \sum_{D \in  \mathcal{D}_h} h_D^r \Vert u_{hp} \Vert_{H^{r+1}(D)} \Vert v_{hp} \Vert_{H^{r+1}(D)} .  
            \end{align*}
            Next, because $u_{hp}|_D$ and $v_{hp}|_D$ are polynomials of degree $p$ as $F_D$ is affine and $r \geq p$, it holds that
            \begin{align*}
                | \langle (A- \widetilde{A}) u_{hp},v_{hp}\rangle| & \lesssim   \|a\|_{W^{r,\infty}(\Omega)} \sum_{D \in  \mathcal{D}_h} h_D^r \Vert u_{hp} \Vert_{H^p(D)} \Vert v_{hp} \Vert_{H^p(D)}.  
            \end{align*}
            By applying an inverse estimate for polynomials \eqref{eq:PolynomialsInverse} and the Cauchy-Schwarz inequality, we obtain
            \begin{align}
                | \langle (A- \widetilde{A}) u_{hp},v_{hp}\rangle | & \lesssim  \|a\|_{W^{r,\infty}(\Omega)} \left(\sum_{D \in  \mathcal{D}_h} h_D^{2(r+1-p)}\Vert u_{hp} \Vert_{H^p(D)}^2\right)^{1/2}  \Vert v_{hp} \Vert_{H^1(\Omega)}. \label{eq:quad_error_A_part1}
            \end{align}
            Let $w_{hp}|_D \in \mathbb{P}_p(D)$ be some local interpolation of $u$ for which there holds
            \begin{align}
                \| u - w_{hp}\|_{H^m(D)} \lesssim  h_D^{t-m}\Vert u\Vert_{H^t(D)}  \label{interpolation_estimate}
            \end{align}
            for all $t,m\in\mathbb{N}$ with $d/2< t\leq p+1$ and $0\leq m\leq t$,  see e.g.~\cite[Chapter~4]{BrennerScott}.
            By triangle inequality, inverse estimate for polynomials and \eqref{interpolation_estimate} we obtain
            \begin{align*}
                \Vert u_{hp} \Vert_{H^p(D)} & \lesssim  h_D^{\min\{0,\lfloor s\rfloor -p\}} \Vert u_{hp}\Vert_{H^{\min\{p,\lfloor s\rfloor\}}(D)} \\
                & \lesssim  h_D^{\min\{0,\lfloor s\rfloor-p\}} \bigg( \Vert u_{hp}-w_{hp}\Vert_{H^{\min\{p,\lfloor s\rfloor\}}(D)} + \Vert w_{hp} -u \Vert_{H^{\min\{p,\lfloor s\rfloor\}}(D)}+\Vert u\Vert_{H^{\min\{p,\lfloor s\rfloor\}}(D)} \bigg) \\
                & \lesssim  h_D^{\min\{0,\lfloor s\rfloor-p\}} \bigg(  h_D^{1-\min\{p,\lfloor s\rfloor\}} \Vert u_{hp}-w_{hp}\Vert_{H^{1}(D)} +   \|u\|_{H^{\min\{p,\lfloor s \rfloor\}}(D)} + \Vert u\Vert_{H^{s}(D)} \bigg) \\
                & \lesssim  h_D^{1-p} \bigg( \Vert u_{hp}-u\Vert_{H^{1}(D)} + \Vert u-w_{hp}\Vert_{H^{1}(D)} \bigg) + h_D^{\min\{0,\lfloor s\rfloor-p\}} \Vert u\Vert_{H^{s}(D)} \\
                &  \lesssim  h_D^{1-p} \bigg( \Vert u_{hp}-u\Vert_{H^{1}(D)} + h_D^{\min\{p,\lfloor s\rfloor\}-1} \|u\|_{H^{\min\{p,\lfloor s \rfloor\}}(D)} \bigg) +  h_D^{\min\{0,\lfloor s\rfloor-p\}} \Vert u\Vert_{H^{s}(D)} \\
                & \lesssim   h_D^{1-p} \Vert u_{hp}-u\Vert_{H^{1}(D)} + h_D^{\min\{0,\lfloor s\rfloor-p\}} \Vert u\Vert_{H^{s}(D)}.
            \end{align*}
            Hence, with \eqref{eq:assumed_error_estimate_uhp} there holds 
            \begin{align*}
                \sum_{D \in  \mathcal{D}_h} h_D^{2(r+1-p)}\Vert u_{hp} \Vert_{H^p(D)}^2 &\lesssim h^{2(r+1-p)} \left( h^{2(1-p)}  \Vert u-u_{hp} \Vert_{H^1(\Omega)}^2 + h^{2\min\{0,\lfloor s\rfloor-p\}}\Vert u \Vert^2_{H^s(\Omega)} \right) \\
                & \lesssim  h^{2(r+1-p)} \left( h^{2(\beta + 1-p)} + h^{2\min\{0,\lfloor s\rfloor-p\}} \right) .
            \end{align*}
            Inserting this estimate into \eqref{eq:quad_error_A_part1} yields
            \begin{align*}
                | \langle (A- \widetilde{A}) u_{hp},v_{hp}\rangle |
                & \lesssim  h^{r+1-p} \left( h^{\beta + 1-p} + h^{\min\{0,\lfloor s\rfloor-p\}} \right)  \Vert v_{hp} \Vert_{H^1(\Omega)}.
            \end{align*}
            That estimate together with \eqref{eq:quad_error_rhs} imply for the total quadrature related error that
            \begin{align*}
                \|(A- \widetilde{A}) u_{hp}-(\ell-\tilde{\ell})\|_{V_{hp}^*} &= \sup_{v_{hp} \in V_{hp} \setminus \{0\}} \frac{ \langle (A- \widetilde{A}) u_{hp}-(\ell-\tilde{\ell}), v_{hp} \rangle }{\|v_{hp}\|_{H^1(\Omega)}} \\
                & \lesssim h^{r+1-p} \left( h^{\beta + 1-p} + h^{\min\{0,\lfloor s\rfloor-p\}}+1 \right) \\
                & \lesssim h^{r+1-p} \left( h^{\beta + 1-p} + h^{\min\{0,\lfloor s\rfloor-p\}} \right) .
            \end{align*}
        \end{proof}
        %
        %
        \begin{remark}
            Clearly, the convergence rate $\beta$ is dependent on the regularity of the solution. For instance, in the two-dimensional case with linear, triangular finite elements and conditions $a\in C^1(\overline{\Omega})$ and $\psi\in H^2(\Omega)$ it was shown in \citep{falk1974error} that $\beta=1$.
            
            For quadratic, triangular finite elements and under the conditions that $ a \equiv 1$, $f\in H^1(\Omega)\cap L^{\infty}(\Omega),$ $\psi\in W^{3,3}(\Omega)\cap W^{2,\infty}(\Omega)$ and assuming that the solution is sufficiently regular, i.e.~$u\in W^{s,k}(\Omega)\text{ with } 1<k<\infty, s <2+\frac{1}{k}$, the rate is $\beta= \frac{3}{2}-\epsilon$ for any $\epsilon >0$ as proven in \citep{WangQuadFEM}.
        \end{remark}
        \begin{remark}
            The expected regularity of the solution $u$ to the obstacle problem is similarly dependent on the regularity of the data. For example, in the case that $a \equiv 1$, $f\in L^2$, $\psi\in H^2(\Omega)$ and $\Omega$ a convex domain or $\partial \Omega\in C^{1,1}$ the solution satisfies $u\in H^2(\Omega)$, see \citep{RodriguesObstacle}.
        \end{remark}
        \begin{remark}
            For $r> \max\{p-1, 2p-2-\beta\}$ the quadrature error has a guaranteed positive convergence rate.
        \end{remark}
        Theorem~\ref{Quad Error} in combination with Corollary~\ref{thm:AprioriErrorPertByApprox} gives us the means to ensure that the use of numerical integration does not decrease the order of convergence.
        \begin{corollary}
            Let $r \geq 2p-1$. Then the perturbed solution $\tilde{u}$ satisfies the a priori estimate
            \begin{align}\label{eq: quadrature perturbed error}
                \| u - \tilde{u}_{hp} \|_V^2\lesssim h^{\beta},
            \end{align}
            for $\beta>0$ as given in Theorem~\ref{Quad Error}.
        \end{corollary}
        \begin{proof}
            We note that classical a priori estimates, see e.g.~\citep{Ciarlet2002, OdenKikuchiPorousMedia,RodriguesObstacle}, give us the bound
            \begin{align}\label{eq:OrderOfConvergence}
                \| u - u_{hp} \|_V^2\lesssim\| u - v_{hp} \|_V^2 + \langle Au- \ell  , v_{hp}-u +v -u_{hp}  \rangle \lesssim h^{\beta}.
            \end{align}
            Next, we differentiate between the cases $\lfloor s\rfloor \geq p$ and $\lfloor s\rfloor\leq p$.
            In the case that $\lfloor s\rfloor \geq p$ and $r \geq 2p-1$ the quadrature error $ \|(A- \widetilde{A}) u_{hp}-(\ell-\tilde{\ell})\|_{V_{hp}^*}$ can also be bounded by $h^\beta$ from above by Theorem~\ref{Quad Error}. Thus, Corollary~\ref{thm:AprioriErrorPertByApprox} and \eqref{eq:OrderOfConvergence} guarantee that \eqref{eq: quadrature perturbed error} holds.
   
            If $\lfloor s\rfloor\leq p$ the convergence rate of the exact finite element-approximation is bounded by $\min\{p,s-1\}=s-1$ anyway. Hence, the order of the quadrature formula $r \geq 2p-1$ is sufficient to maintain a convergence rate of $\beta$ for the quadrature error as well, which combined with Corollary~\ref{thm:AprioriErrorPertByApprox} and \eqref{eq:OrderOfConvergence} yields \eqref{eq: quadrature perturbed error}. 
        \end{proof}

\section{Perturbation introduced during discretization} \label{sec:perturbed_by vareq}
    To further illustrate the applicability of our abstract result, we discuss another kind of perturbation, which is caused during the process of discretization. Our discussed examples are variational inequalities constrained by a variational equality. 
    By a variational equality constrained variational inequality we understand the problem of finding a pair $(u,\lambda) \in K \times \Lambda$ such that
    \begin{subequations} \label{eq:VI_plus_VE}
        \begin{alignat}{6}
            &\langle D u ,v-u \rangle& \,+\,& \langle B^\top \lambda, v-u \rangle & &\geq \langle f,v-u \rangle &\quad&  \label{eq:VI_plus_VE_A} \\
            &\langle B u,\mu \rangle& \,-\, &\langle C \lambda,\mu \rangle & &= \langle g, \mu \rangle &\quad&  \label{eq:VI_plus_VE_B}
        \end{alignat}
    \end{subequations}
    for all $ v \in K$ and $\mu \in \Lambda$, respectively.
    
    Here, $K$ is a non-empty, closed and convex subset of a real Hilbert space $V$ and $\Lambda$ is another real Hilbert space. Furthermore, let $f \in V^*$ and $g \in \Lambda^*$. The operators $D:V \rightarrow V^*$, $B:V \rightarrow \Lambda^*$, and $C:\Lambda \rightarrow \Lambda^*$ are assumed to be continuous i.e.
    \begin{align*}
        \norm{Dv}_{V^*}\leq c_D\norm{v}_V,\quad
        \norm{Bv}_{\Lambda^*}\leq c_B\norm{v}_V,\quad
        \norm{C\mu}_{\Lambda^*}\leq c_C\norm{\mu}_{\Lambda},
    \end{align*}
    for all $v\in V$, $\mu\in\Lambda$ and some constants $c_D, c_B, c_C >0$. Moreover, $D$ and $C$ are assumed to be $V$-elliptic, $\Lambda$-elliptic, respectively, i.e.
    \begin{align*}
        \langle Dv,v\rangle \geq \alpha_D\norm{v}^2_V,\quad \langle C\mu,\mu\rangle \geq \alpha_C \norm{\mu}^2_{\Lambda}
    \end{align*}
    for all $v\in V$, $\mu\in\Lambda$ and some constants $\alpha_D, \alpha_C >0$. Furthermore, the operator $C$ is assumed to be symmetric. The operator $B^\top: \Lambda \rightarrow V^*$ denotes the adjoint operator to $B$, which is continuous with continuity constant $c_{B^\top}>0$, see e.g.~\citep{YosidaFunctionalAnalysis}.
    
    Note that \eqref{eq:VI_plus_VE} is closely related to a $C$-perturbed saddle point problem \citep{Brezzi1991,Hong2023} and cannot straight away be viewed as an example for the variational inequality \eqref{eq:Varineq} with $A$ and $\ell$ defined as
    $$
        \begin{pmatrix} D & B^\top \\ B &-C \end{pmatrix} \quad \text{and} \quad \begin{pmatrix} f\\g \end{pmatrix}
    $$
    as this operator $A$ would be indefinite and not elliptic.
    
    Instead, we define $A:=D + B^\top C^{-1} B$ and $\ell:=f+B^\top C^{-1} g$. From \eqref{eq:VI_plus_VE_B} we find that $\lambda = C^{-1}(Bu-g)$ and, thus, we may obtain the variational inequality problem of finding a $u \in K$ such that
    \begin{align} \label{eq:reducedVI}
        \langle Au,v-u \rangle \geq \langle \ell, v-u \rangle 
    \end{align}
    for all $v \in K$.
    
    Note that the problems \eqref{eq:VI_plus_VE} and \eqref{eq:reducedVI} are equivalent: If a pair $(u,\lambda)\in K\times\Lambda$ solves \eqref{eq:VI_plus_VE}, then $u\in K$ solves \eqref{eq:reducedVI}. Conversely, if $u\in K$ solves \eqref{eq:reducedVI}, then the pair $(u,\lambda)\in K\times\Lambda$ with $\lambda := C^{-1}(Bu-g)$ solves \eqref{eq:VI_plus_VE}.
    
    We note that the operator $C^{-1}$ is continuous and $\Lambda^*$-elliptic, i.e.
    \begin{align*}
        \norm{C^{-1}\varphi}_{\Lambda}\leq \frac{1}{\alpha_C}\norm{\varphi}_{\Lambda^*}, \quad \langle C^{-1}\varphi,\varphi\rangle \geq \frac{1}{c_C}\norm{\varphi}^2_{\Lambda^*}
    \end{align*}
    for all $\varphi\in\Lambda^*$, see \cite[Chapter~3]{steinbach2008numerical}.
    \begin{theorem}\label{prop:A_cont_ellip}
        The mapping $A:V\to V^*$ is continuous and $V$-elliptic i.e.
        \begin{align*}
            \norm{Av}_{V^*}\leq \Big(c_D+\frac{c_B^2}{\alpha_C}\Big) \norm{v}_V\quad\text{ and }\quad \langle Av,v\rangle \geq \alpha_D\norm{v}_V^2
        \end{align*}
        for all $v\in V$.
    \end{theorem}
    \begin{proof}
        Let $v,w\in V$. Then, we conclude from the continuity of $D,B$ and $C^{-1}$ that
        \begin{align*}
            |\langle Av,w\rangle| &\leq |\langle Dv,w\rangle| + | \langle B^\top C^{-1}Bv,w\rangle|\\
            &\leq \norm{Dv}_{V^*}\,\norm{w}_V+| \langle C^{-1}Bv,Bw\rangle|\\
            &\leq  c_D\norm{v}_V\,\norm{w}_V + \norm{C^{-1}Bv}_{\Lambda}\,\norm{Bw}_{\Lambda^*}\\
            &\leq  c_D\norm{v}_V\,\norm{w}_V+\frac{1}{\alpha_C}\norm{Bv}_{\Lambda^*}\,\norm{Bw}_{\Lambda^*}\\
            & \leq \Big(c_D+\frac{c_B^2}{\alpha_C}\Big) \norm{v}_V\,\norm{w}_V.
        \end{align*}
        Thus, $A$ is continuous.
        
        Applying the $V$-ellipticity of $D$ and the $\Lambda^*$-ellipticity of $C^{-1}$ we obtain
        \begin{align*}
            \langle Av,v\rangle &= \langle Dv,v\rangle+ \langle B^\top C^{-1} Bv,v\rangle \geq  \alpha_D\norm{v}_V^2+ \langle C^{-1} Bv,Bv\rangle \geq  \alpha_D\norm{v}_V^2
        \end{align*}
        for any $v\in V$. Thereby, $A$ is $V$-elliptic.
    \end{proof}
    By Theorem~\ref{prop:A_cont_ellip} the problem \eqref{eq:reducedVI} is indeed an example for \eqref{eq:Varineq}. For a Galerkin discretization, let $\widetilde{V}\subset V$ and $\widetilde{\Lambda}\subset\Lambda$ be two finite dimensional subsets. We may replace $K \times \Lambda$ by some $\widetilde{K} \times \widetilde{\Lambda} \subset \widetilde{V} \times \widetilde{\Lambda}$, where $\widetilde{K}$ is a non-empty, closed and convex subset of $\widetilde{V}$ and not necessarily of $K$. Thus, we seek a pair $(\tilde{u},\tilde{\lambda}) \in \widetilde{K} \times \widetilde{\Lambda}$ such that
    \begin{subequations} \label{eq:Discrete_VI_plus_VE}
        \begin{alignat}{6}
            &\langle D \tilde{u} ,\tilde{v}-\tilde{u} \rangle& \,+\,& \langle B^\top \tilde{\lambda}, \tilde{v}-\tilde{u} \rangle & &\geq \langle f,\tilde{v}-\tilde{u} \rangle &\quad&   \label{eq:Discrete_VI_plus_VE_A}\\
            &\langle B \tilde{u},\tilde{\mu} \rangle& \,-\, &\langle C \tilde{\lambda},\tilde{\mu} \rangle & &= \langle g, \tilde{\mu} \rangle &\quad&  \label{eq:Discrete_VI_plus_VE_B}
        \end{alignat}
    \end{subequations}
    for all $\tilde{v} \in \widetilde{K}$ and $\tilde{\mu} \in \widetilde{\Lambda}$, respectively.
    
    As before, we need to condense out $\tilde{\lambda}$. For that, let $\widetilde{R}: \Lambda^* \rightarrow C \widetilde{\Lambda}$ be a linear mapping such that
    \begin{align}\label{eq:RM_Property}
        \langle \widetilde{R} g, \tilde{\mu} \rangle = \langle g,\tilde{\mu} \rangle
    \end{align}
    for all $\tilde{\mu} \in \widetilde{\Lambda}$ and $ g \in \Lambda^*$,
    whose existence is guaranteed by the Lax-Milgram Lemma, see e.g.~\cite[Theorem~2.7.7.]{BrennerScott}.
    Therewith, we may define 
    \begin{align*}
        \widetilde{C}^{-1} := C^{-1}\widetilde{R} :\Lambda^* \rightarrow \widetilde{\Lambda}
    \end{align*}
    and thus obtain $\tilde{\lambda} = \widetilde{C}^{-1}(B\tilde{u}-g)$ as the solution to \eqref{eq:Discrete_VI_plus_VE_B}.
    
    Hence, \eqref{eq:Discrete_VI_plus_VE} reduces to: Find a $\tilde{u} \in \widetilde{K}$ such that
    \begin{align} \label{eq:Discrete_reducedVI}
        \langle \widetilde{A}\tilde{u},\tilde{v}-\tilde{u} \rangle \geq \langle \tilde{\ell}, \tilde{v}-\tilde{u} \rangle 
    \end{align}
    for all $\tilde{v} \in \widetilde{K}$ with $\widetilde{A} := D + B^\top \widetilde{C}^{-1} B$ and $\tilde{\ell} := f+B^\top \widetilde{C}^{-1} g$. The problems \eqref{eq:Discrete_VI_plus_VE} and \eqref{eq:Discrete_reducedVI} are equivalent in the same sense as their continuous counterparts.
    \begin{lemma}\label{lem:CM-1_Property}
        The operator $\widetilde{C}^{-1}$ is  continuous i.e.
        \begin{align*}
             \Vert \widetilde{C}^{-1}\varphi\Vert_{\Lambda}\leq \frac{1}{\alpha_C}\Vert \varphi \Vert_{\Lambda^*}
        \end{align*}
        for all $\varphi\in\Lambda^*$ and positive semi-definite.
    \end{lemma}
    \begin{proof}
        Let $\varphi\in\Lambda^*$. Then due to the $\Lambda$-ellipticity of $C$ we have
        \begin{align*}
            \Vert \widetilde{C}^{-1}\varphi\Vert_{\Lambda}^2&\leq \frac{1}{\alpha_C} \langle C(\widetilde{C}^{-1}\varphi),\widetilde{C}^{-1}\varphi\rangle=\frac{1}{\alpha_C}\langle \widetilde{R}\varphi,\widetilde{C}^{-1}\varphi\rangle.
        \end{align*}
        Then, by \eqref{eq:RM_Property} and $\widetilde{C}^{-1}\varphi\in\widetilde{\Lambda}$ it follows that
        \begin{align*}
            \Vert \widetilde{C}^{-1}\varphi\Vert_{\Lambda}^2&\leq \frac{1}{\alpha_C}\langle \varphi,\widetilde{C}^{-1}\varphi\rangle\leq \frac{1}{\alpha_C}\Vert \varphi \Vert_{\Lambda^*}\Vert \widetilde{C}^{-1}\varphi\Vert_{\Lambda}.
        \end{align*}
        Therefore, $\widetilde{C}^{-1}$ is continuous with continuity constant $1/\alpha_C$, which is independent of the discretization. Furthermore, for arbitrary $\varphi\in\Lambda^*$ it follows from \eqref{eq:RM_Property} that
        \begin{align*}
            \langle\varphi,\widetilde{C}^{-1}\varphi\rangle &= \langle \widetilde{R}\varphi,\widetilde{C}^{-1}\varphi\rangle = \langle C(\widetilde{C}^{-1}\varphi),\widetilde{C}^{-1}\varphi\rangle
        \end{align*}
        which then combined with the $\Lambda$-ellipticity of $C$ leads to
        \begin{align*}
            \langle\varphi,\widetilde{C}^{-1}\varphi\rangle &\geq \alpha_C\norm{\widetilde{C}^{-1}\varphi}_{\Lambda}^2\geq 0.
        \end{align*}
        Therefore, the operator $\widetilde{C}^{-1}$ is positive semi-definite.
    \end{proof}
    \begin{theorem}\label{thm: pert. A cont ell}
        The mapping  $\widetilde{A}:\widetilde{V}\to \widetilde{V}^*$ is continuous and $\widetilde{V}$-elliptic i.e.
        \begin{align*}
            \norm{\widetilde{A}\tilde{v}}_{\widetilde{V}^*}\leq \Big(c_D+\frac{c_B^2}{\alpha_C}\Big) \norm{\tilde{v}}_{\widetilde{V}}\quad\text{ and }\quad\langle \widetilde{A}\tilde{v},\tilde{v}\rangle \geq  \alpha_D\norm{\tilde{v}}_{\widetilde{V}}^2
        \end{align*}
        for all $\tilde{v}\in\widetilde{V}$.
    \end{theorem}
    \begin{proof}
        For any $\tilde{v},\tilde{w}\in \widetilde{V}$ it holds due to the continuity of $D,B$ and $\widetilde{C}^{-1}$ that
        \begin{align*}
            |\langle\widetilde{A}\tilde{v},\tilde{w}\rangle| &\leq |\langle D\tilde{v},\tilde{w}\rangle| + | \langle B^\top \widetilde{C}^{-1}B\tilde{v},\tilde{w}\rangle|\leq \norm{D\tilde{v}}_{\widetilde{V}^*}\,\norm{\tilde{w}}_{\widetilde{V}}+| \langle \widetilde{C}^{-1}B\tilde{v},B\tilde{w}\rangle|\\
            &\leq  c_D\norm{\tilde{v}}_{\widetilde{V}}\,\norm{\tilde{w}}_{\widetilde{V}} + \norm{\widetilde{C}^{-1}B\tilde{v}}_{\widetilde{\Lambda}}\,\norm{B\tilde{w}}_{\widetilde{\Lambda}^*}\\
            &\leq  c_D\norm{\tilde{v}}_{\widetilde{V}}\,\norm{\tilde{w}}_{\widetilde{V}}+\frac{1}{\alpha_C}\norm{B\tilde{v}}_{\widetilde{\Lambda}^*}\,\norm{B\tilde{w}}_{\widetilde{\Lambda}^*}\\
            & \leq \Big(c_D+\frac{c_B^2}{\alpha_C}\Big) \norm{\tilde{v}}_{\widetilde{V}}\,\norm{\tilde{w}}_{\widetilde{V}},
        \end{align*}
        which implies the claimed continuity of $\widetilde{A}$.
        
        Then due to the positive semi-definiteness of $\widetilde{C}^{-1}$ and the $V$-ellipticity of $D$ we obtain
        \begin{align*}
            \langle \widetilde{A}\tilde{v},\tilde{v}\rangle &= \langle D\tilde{v},\tilde{v}\rangle +  \langle B^\top \widetilde{C}^{-1}B\tilde{v},\tilde{v}\rangle
            \geq  \alpha_D\norm{\tilde{v}}_{\widetilde{V}}^2 + \langle \widetilde{C}^{-1}B\tilde{v},B\tilde{v}\rangle
            \geq  \alpha_D\norm{\tilde{v}}_{\widetilde{V}}^2
        \end{align*}
        for any $\tilde{v}\in \widetilde{V}$.
    \end{proof}
    Theorem~\ref{thm: pert. A cont ell} allows us to view \eqref{eq:Discrete_reducedVI} as  an example of \eqref{eq:PerturbedVarIneq}.
    Hence, we can apply Theorem~\ref{thm:abstractStrangFalk} to obtain an a priori error estimate for the Galerkin solution $(\tilde{u},\tilde{\lambda})$.
    \begin{theorem} \label{thm:aprioriErrorPertByEQ}
        There exist constants $C_1,C_2,C_3>0$ such that
        \begin{align*}
            \| u - \tilde{u} \|_V^2  + \| \lambda - \tilde{\lambda} \|_\Lambda^2  \leq  C_1\| u - \tilde{v} \|_V^2 + C_2\|\lambda - \tilde{\mu}\|_\Lambda^2 + C_3\langle Du + B^\top \lambda -f  , \tilde{v}-u +v -\tilde{u}  \rangle 
        \end{align*}
        for all $\tilde{v} \in \widetilde{K}$, $v \in K$ and $\tilde{\mu} \in \widetilde{\Lambda}$.
    \end{theorem}
    \begin{proof}
        Direct computations give $Au-\ell = Du + B^\top \lambda -f$. Let $\widetilde{G} := \widetilde{C}^{-1}C$ be the Galerkin projection operator. With $\lambda = C^{-1}(Bu-g)$, see \eqref{eq:VI_plus_VE_B}, we obtain
        \begin{align*}
            \ell - \tilde{\ell}&= B^\top C^{-1} g - B^\top \widetilde{C}^{-1} g =  B^\top\left(-\lambda + C^{-1}B u\right) - B^\top \widetilde{C}^{-1} \left(-C\lambda + B u \right)  \\
            &= B^\top \left(\widetilde{G}-I\right) \lambda + B^\top \left(C^{-1}-\widetilde{C}^{-1} \right) B u.
            \end{align*}
        Hence,
        \begin{align*}
            (A- \widetilde{A}) \tilde{v}-(\ell-\tilde{\ell}) &= B^\top \left(C^{-1} -\widetilde{C}^{-1} \right)B \tilde{v}  - (\ell-\tilde{\ell}) \\
            & = B^\top \left(C^{-1} -\widetilde{C}^{-1} \right)B \left(\tilde{v} - u \right) + B^\top \left(I-\widetilde{G}\right) \lambda
        \end{align*}
        and, therewith, 
        \begin{align*}
            \|(A- \widetilde{A}) \tilde{v}-(\ell-\tilde{\ell})\|_{V^*} &= \|B^\top \left(C^{-1} -\widetilde{C}^{-1} \right)B \left(\tilde{v} - u \right) + B^\top \left(I-\widetilde{G}\right) \lambda\|_{V^*} \\
            & \leq \|B^\top C^{-1} B \left(\tilde{v} - u \right)\|_{V^*}+\|B^\top  \widetilde{C}^{-1}  B \left(\tilde{v} - u \right)\|_{V^*} + \|B^\top \left(I-\widetilde{G}\right) \lambda\|_{V^*} \\
            & \leq 2\Big(c_{B^\top}+\frac{c_B}{\alpha_C}\Big)\|u-\tilde{v} \|_V + c_{B^\top}\| \left(I-\widetilde{G}\right) \lambda\|_\Lambda
            \\
            & \leq 2\Big(c_{B^\top}+\frac{c_B}{\alpha_C}\Big)\|u-\tilde{v} \|_V + c_{B^\top}\frac{c_C}{\alpha_C}\|\lambda - \tilde{\mu}\|_\Lambda
        \end{align*}
        for all $\tilde{\mu} \in \widetilde{\Lambda}$, where the last inequality follows with C\'ea's Lemma, see e.g.~\cite[Theorem 2.8.1]{BrennerScott}.
        
        Next, we apply Theorem~\ref{thm:abstractStrangFalk} to \eqref{eq:reducedVI} and \eqref{eq:Discrete_reducedVI} to obtain
        \begin{align*}
            \| u - \tilde{u} \|_V^2 & \leq \Bigg(2+\frac{4(c_D\alpha_C+c_B^2)^2}{(\alpha_C\alpha_D)^2}\Bigg)          \| u - \tilde{v} \|_V^2 + \frac{4}{\alpha_D}\langle Au -\ell  , \tilde{v}-u +v -\tilde{u}  \rangle  +  \frac{4}{\alpha_D^2}\|(A- \widetilde{A}) \tilde{v}-(\ell-\tilde{\ell})\|_{V^*}^2  \\
            & \leq    C_0\| u - \tilde{v} \|_V^2 + \frac{c_{B^\top}c_C}{\alpha_C}\|\lambda - \tilde{\mu}\|_\Lambda^2 + \frac{4}{\alpha_D^2}\langle Du + B^\top \lambda -f  , \tilde{v}-u +v -\tilde{u}  \rangle 
        \end{align*}
        for all $\tilde{v} \in \widetilde{K}$, $v \in K$ and $\tilde{\mu} \in \widetilde{\Lambda}$ with 
        \begin{align*}
            C_0:= \Bigg(2+\frac{4(c_D\alpha_C+c_B^2)^2}{(\alpha_C\alpha_D)^2}+\frac{2(c_{B^\top}\alpha_C+c_B)}{\alpha_C}\Bigg).
        \end{align*}
        Using the Strang Lemma \cite[Theorem 8.30]{HackbuschStrang} on \eqref{eq:VI_plus_VE_B} and \eqref{eq:Discrete_VI_plus_VE_B} we immediately obtain
        \begin{align*}
            \| \lambda - \tilde{\lambda} \|_\Lambda \leq\Big(1+\frac{c_C}{\alpha_C}\Big)\|\lambda - \tilde{\mu} \|_\Lambda + \frac{1}{\alpha_C}\|B(u-\tilde{u}) \|_{\Lambda^*} \leq \Big(1+\frac{c_C}{\alpha_C}\Big)\|\lambda - \tilde{\mu} \|_\Lambda + \frac{c_B}{\alpha_C}\|u-\tilde{u} \|_V.
        \end{align*}
        Finally, combining the last two estimates completes the proof  with 
        \begin{align*}
            C_1:= \Bigg(1+\frac{2c_B}{\alpha_C}\Bigg)C_0,\quad
            C_2:= \Bigg(1+\frac{2c_B}{\alpha_C}\Bigg)\Bigg(2+\frac{2c_C+c_{B^\top}c_C}{\alpha_C}\Bigg),\quad
            C_3:=\Big(1+\frac{2c_B}{\alpha_C}\Big)\frac{4}{\alpha_D^2}.
        \end{align*}
    \end{proof}
    %
    %
    \subsection{Example: Biot contact problem}
        The Biot contact problem discussed in \citep{BANZ2024219} belongs to the field of poroelasticity and can be used to model, for example, a human knee joint, see \citep{sahu2016}.
        
        Let $\Omega \subset \mathbb{R}^d$ with $d=2,3$ be a bounded Lipschitz domain. The boundary $\partial \Omega$ is decomposed in two ways: One decomposition is $\partial \Omega$ = $\Gamma_f \cup \Gamma_p$, where $\Gamma_f$ is the Neumann part and $\Gamma_p$ the Dirichlet part corresponding to the fluid pressure $p$. The other decomposition is $\partial \Omega$ = $\Gamma_t \cup \Gamma_d \cup \Gamma_c$, where $\Gamma_t$ is the Neumann part, $\Gamma_d$ is the Dirichlet part and $\Gamma_c$ is the contact part corresponding to the displacement $u$. Furthermore, both $\Gamma_p$ and $\Gamma_d$ have positive measure and $\overline{\Gamma_d}\cap \overline{\Gamma_c} = \emptyset$.
        The outer unit normal on $\partial \Omega$ is denoted by $n$ and $g:\Gamma_C \rightarrow \mathbb{R}$ is the gap function to the obstacle in normal direction of $\Gamma_c$. 
        A weak formulation of that problem is given by \eqref{eq:VI_plus_VE} with
        \begin{align*}
            \Lambda &:= H_{0, \Gamma_{p}}^{1}(\Omega)=\left\{ v \in H^{1}(\Omega):\left.v\right|_{\Gamma_{p}}=0\right\}, \\
            V &:= \left[H_{0, \Gamma_{d}}^{1}(\Omega)\right]^{d},\\
            K &:= \{ v \in V \ : \ v\cdot n \leq g \text{ on } \Gamma_c\},  \\
            \langle D u,v \rangle &:= \int_\Omega 2\,\tau\, \epsilon(u) :  \epsilon(v) + \iota\, \operatorname{div}(u) \operatorname{div}(v) \, dx, \\
            \langle Bu,q \rangle &:= \int_\Omega \alpha\, \operatorname{div}(u) q\, dx,\\
            \langle C p, q \rangle &:= \int_\Omega \frac{\alpha^{2}}{\iota} p\, q + (\kappa \nabla p)^\top \nabla q \, dx ,
        \end{align*}
        where $\epsilon(v) := (\nabla v + \nabla^\top v)/2$ is the linearized strain tensor, $\iota>0$ and $\tau>0$ are the first and second Lamé-parameters, $\alpha>0$ is the Biot–Willis constant and $\kappa>0$ a symmetric uniformly positive defined matrix describing the permeability.
        %
    %
    %
    \subsection{Example: BEM for Signorini contact problem}
        Let $\Omega\subset\mathbb{R}^d$ with $d=2,3$ be a bounded Lipschitz domain. Its boundary $\partial\Omega$ is decomposed into the disjoint Dirichlet part $\Gamma_D$, Neumann part $\Gamma_N$ and contact part $\Gamma_C$, where $\Gamma_C$ is assumed to have positive measure. The (simplified) Signorini contact problem is described by the search for a $u$ fulfilling
        \begin{subequations}
            \begin{alignat*}{2}
                -\Delta u &= 0 &\quad& \text{in } \Omega\\
                u &= 0 &\quad& \text{on } \Gamma_D\\
                \frac{\partial u}{\partial n} &= f &\quad& \text{on } \Gamma_N\\
                u \leq g,\: \frac{\partial u}{\partial n} \leq 0,\:  (u-g)\frac{\partial u}{\partial n} &= 0 &\quad& \text{on } \Gamma_C,
            \end{alignat*}
        \end{subequations}
        for some gap function $g:\Gamma_C\to\mathbb{R}$ and Neumann data $f$ (see \citep{Gwinner2018} and the references therein).
        Its boundary integral formulation is given by \eqref{eq:VI_plus_VE} with the hypersingular integral operator with $D:=\mathcal{W}$, the single-layer integral operator $C:= \mathcal{V}$ and $B:=I+\mathcal{K}$, where $I$ is the identity operator and $\mathcal{K}$ the double-layer integral operator. These operators are defined in e.g.~\citep{MAISCHAK2005} and are linear and continuous, with both $D$ and $C$ elliptic, see \citep{Costabel1988,COSTABEL1985}.
        
        Here $\Lambda:= H^{-1/2}(\partial \Omega)$ and 
        $$
            K := \left\{ v \in H^{1/2}(\partial \Omega):  \, v|_{\Gamma_D}=0,\ v|_{\Gamma_C}\leq g \right\}
        $$
        are the appropriate sets to determine the normal derivative $\lambda=\frac{\partial u}{\partial n}$ on $\partial \Omega$ and the trace of $u$. We obtain the popular symmetric Poincar\'e-Steklov Operator
        $$
            S:=\mathcal{W} + (I+\mathcal{K})^\top \mathcal{V}^{-1} (I+\mathcal{K})= D+B^\top C^{-1} B = A.
        $$
        Discretizing the Poincar\'e-Steklov formulation necessarily results in discretizing $S$ itself as an explicit representation of $\mathcal{V}^{-1}$ is not known, and we may use Theorem~\ref{thm:aprioriErrorPertByEQ} to recover well-known a priori error estimate results (see e.g.~\citep{Gwinner2018}).

        %
    %
    \subsection{Example: Optimal control problem}
        A typical optimal control problem that falls within the framework of Section~\ref{sec:perturbed_by vareq} is the distributed elliptic control problem with control constraints, see e.g.~\citep{Banz2022OptimalControl, Hintermueller2008, Li2002}.  That is 
        \begin{subequations}
            \begin{alignat*}{2}
                &\text{minimize}& \quad & \frac{1}{2}\, \| y-y^d\|_{L^2(\Omega)}^2 + \frac{\alpha}{2}\,\| u-u^d\|_{L^2(\Omega)}^2\\
                &\text{over}&\quad & (y,u)\in H^1_0(\Omega)\times L^2(\Omega)\\
                &\text{subject to} &\quad & -\Delta y = f+u\\
                &&\quad & u \leq \psi \text{ a.e. in $\Omega$}
            \end{alignat*}
        \end{subequations}
        for some given data $u^d,y^d,f \in L^2(\Omega)$, $\psi \in C^0(\Omega)$ and bounded Lipschitz domain $\Omega$.
        
        Its first order optimality condition is indeed to find a triple $(u,y,p) \in K \times H^1_0(\Omega) \times H^1_0(\Omega)$ such that
        \begin{subequations}\label{eq:foc}
            \begin{alignat}{2}
                \int_\Omega u\, (z-u) \, dx -\int_\Omega \alpha^{-1}p \, (z-u) \, dx  &\geq \int_\Omega u^d  (z-u) \, dx &\quad&   \\
                - \int_\Omega u\, v \, dx + \int_\Omega (\nabla y)^\top \nabla v \, dx  &=\int_\Omega f\, v \, dx &\quad&   \\
                \int_\Omega y\, w \, dx + \int_\Omega (\nabla p)^\top \nabla w \, dx   &=\int_\Omega y^d w \, dx &\quad&  
            \end{alignat}
        \end{subequations}
        for all $z \in K$, $v \in H^1_0(\Omega)$ and $w \in H^1_0(\Omega)$, respectively, where
        $$
            K:= \{ v \in L^2(\Omega) \mid v \leq \psi \text{ a.e.~in } \Omega \}.
        $$
        The problem \eqref{eq:foc} is a slight extension of \eqref{eq:VI_plus_VE} as it contains two instead of one 
        variational equations but can be analyzed similarly, see~\citep{Banz2022OptimalControl}.

\section{Numerical experiments}\label{sec:numer_exp}
    The following numerical experiments aim to validate the theoretical convergence rates of Theorem~\ref{Quad Error} and also \eqref{eq: quadrature perturbed error} even under weaker restrictions on the element shapes. In particular, we do not require the mappings $F_D$ to be affine but allow it to be bi-polynomial in these numerical experiments. Additionally, we examine several discretization methods, namely uniform and adaptive $h$-versions with their different convergence rates as well as a uniform $p$-version to analyze the $p$-dependencey of the constants may affect the convergence rates.
    \subsection{Problem description}
    We solve the obstacle problem \eqref{eq:obstacleVI} with the data $\Omega:=\{x \in \mathbb{R}^2 : \norm{x} \leq 1.5\}$, $f := -2$, $\psi :=\log(1.5)-5/8$ and $a:=1$. Its exact solution is known to be
    \begin{align*}
        u(x) =
        \begin{cases}
            \psi, &\norm{x} \leq 1, \\
            \left(\norm{x}^2 - \log(\norm{x}^2)-1\right)/2 + \psi, & \norm{x} \geq 1. 
        \end{cases}
    \end{align*}
    The initial mesh, of which all other meshes are generated by some isotropic refinements, is depicted in Figure~\ref{fig:initmesh} next to the exact solution $u$. 
    The mapping $F_D: \hat{D}\to D$, defined on the reference square $\hat{D}:=[-1,1]^2$, is linear for those elements with an $l^\infty$ norm of all four vertices less than $2^{-1/2}$, blue-colored elements in Figure~\ref{fig:initmesh}. For elements touching the boundary $\partial \Omega$, marked yellow in Figure~\ref{fig:initmesh}, the mapping $F_D$ is a polynomial of degree 6 in each direction such that the domain approximation is always less than the machine precision; see~\cite{Banz2022Bingham} where such a domain approximation and such a mapping are also used, and discussed in more detail. For the remaining elements, $F_D$ is bilinear. Note that the free boundary is a circle of radius one and center zero. The singularity of $u$ is located on this free boundary and is thus always contained within the elements for which the mappings $F_D$ are bilinear but not linear.
    \subsection{Convergence rates}
    As $\nabla F_D$ is not always constant, we have to integrate rational functions for the computations of the local stiffness matrices. Using a Gaussian quadrature based on the tensor product with $q \times q$ points per element, which is of order $2q$, the integration will be inexact leading to the perturbed operator $\widetilde{A}$. Likewise, the integration for the right-hand side can be inexact despite that $f=-2$ is constant.
    
    The solver for discrete problems is a primal-dual active set solver and computes some $\tilde{u}_{hp}:=u_{hp,q}$. For the h-adaptive schemes we use the optimal but usually unavailable exact $H^1(\Omega)$ error $|u-u_{hp,p+11}|_{H^1(\Omega)}$ as the error estimator to steer the refinement and we use Dörfler marking with bulk parameter $\theta = \frac{1}{2}$.
    \begin{figure}[htbp]
        \centering 
        \subfloat[Initial mesh with curved elements]{
        \hspace*{4pt}%
  	    \includegraphics[trim = 0mm 5mm 5mm 0mm, clip, keepaspectratio,scale=0.1285]{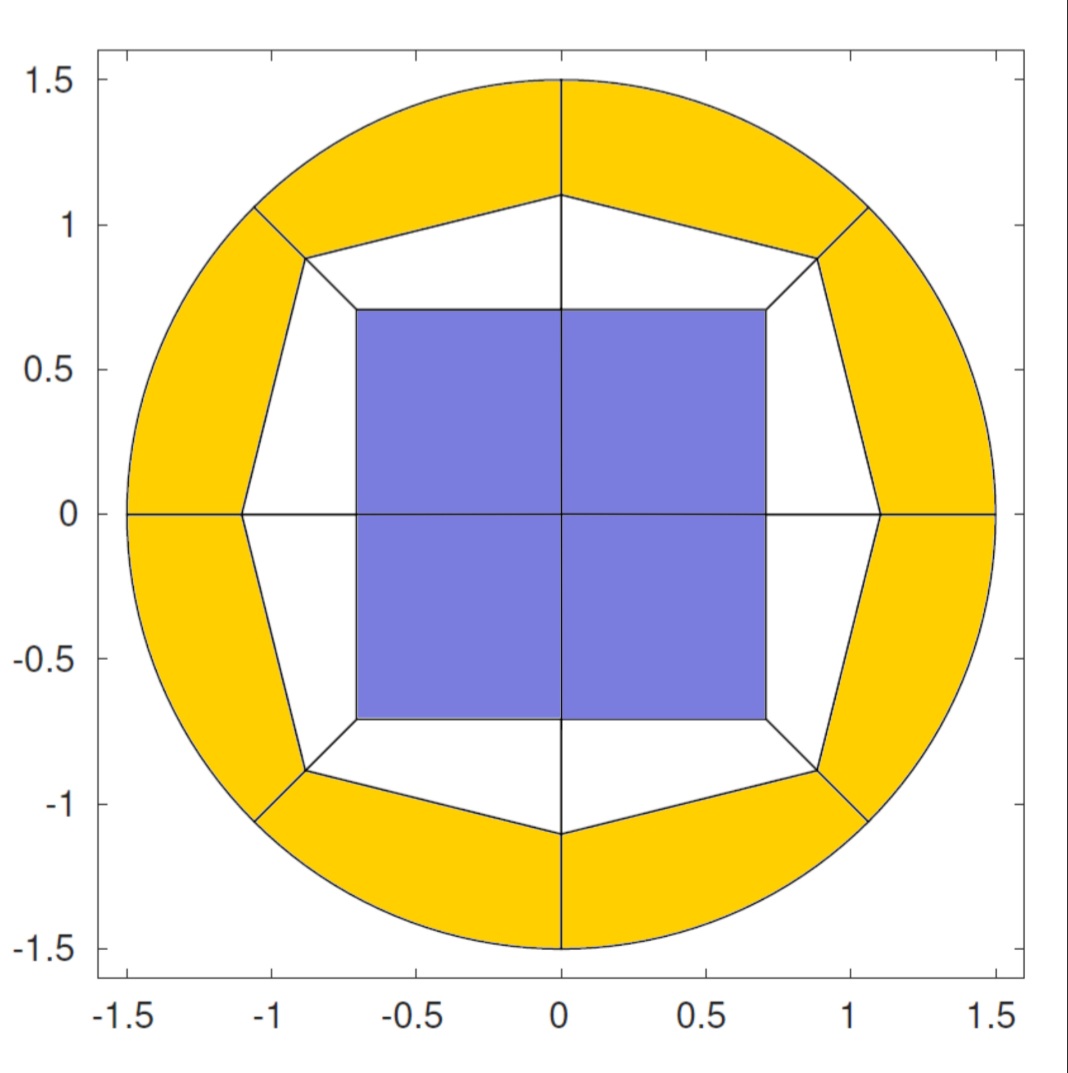}  }
        \hspace*{2.5em}
        \subfloat[Exact solution $u$]{
	    \includegraphics[trim = 4mm 5mm 9mm 5mm, clip, keepaspectratio,scale = 0.482]{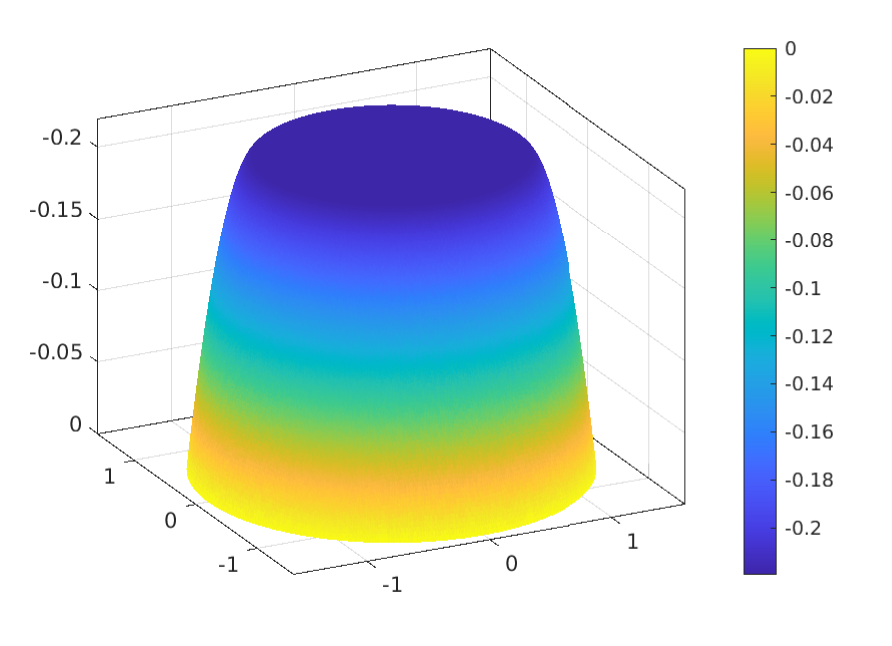}  }
        \caption{Initial mesh and exact solution}
        \label{fig:initmesh}
    \end{figure}
    If $q\leq p-1$ where $p$ is the uniform polynomial degree, then the global stiffness matrix computed in the numerical experiments exhibits multiple zero eigenvalues and is thus no longer positive definite, as the assumptions of Theorem~\ref{pertTh} are not fulfilled. For the minimal number of quadrature points $q=p$ we observe severe under-integration which significantly decreases the smallest but not the largest eigenvalues of the computed stiffness matrix. However, this effect on the eigenvalues rapidly decays with increasing $q$.
    
    \indent Figure~\ref{fig:error_VI} shows the $H^1$-error $|u-\tilde{u}_{hp}|_{H^1(\Omega)}=|u-u_{hp,p+11}|_{H^1(\Omega)}$ for $q=p+11$ for different discretizations to establish the actual convergence rates needed in Theorem~\ref{Quad Error}. These are uniform $h$-versions with $p=1,2,3$, uniform $p$-version with 80 elements, and $h$-adaptive versions with $p=1,2,3$.
    Note that $u \in H^{5/2-\epsilon}(\Omega)$ for any $\epsilon>0$ and thus the experimental convergence rates with respect to the degrees of freedom $N$ are as expected by the approximation properties of the finite element spaces, see the last row in Table~\ref{tab:eoc_error}. The $h$-adaptive schemes with $p>3$ and $hp$-adaptive schemes will also only exhibit a convergence rate of 3/2 with respect to the degrees of freedom as isotropic mesh refinement has limited capability to resolve a curved free boundary which is in this case a circle and the location of the singularity, see e.g.~\cite{BANZ2015}.
    \begin{figure}[htb!]
	   \centering
      \includegraphics[scale=0.47]{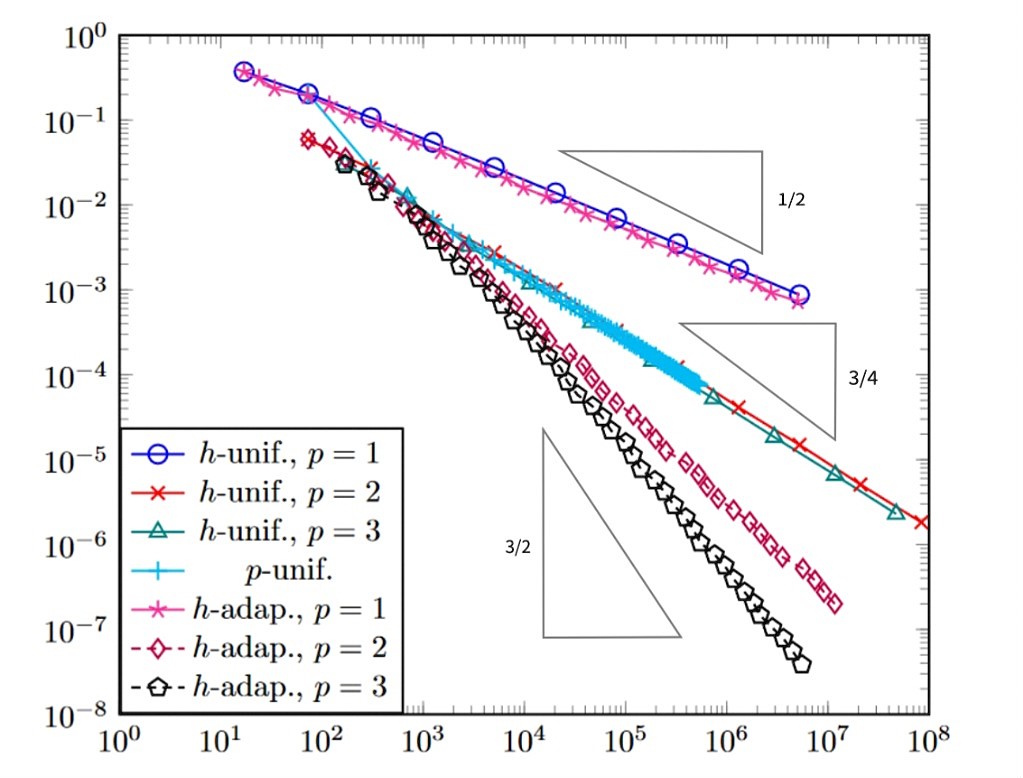}
	   \caption{Approximation error $|u-u_{hp,p+11}|_ {H^1(\Omega)}$ vs.~degrees of freedom for different discretization methods} \label{fig:error_VI}
    \end{figure}
    Next, we repeat these computations but with $q=p+j$ for $j \in \{0,1,2,3,4,5\}$ instead of $q=p+11$ and summarize exemplarily some achieved convergence rates in Table~\ref{tab:eoc_error}. Clearly, the rates are not affected by the quadrature error, neither for uniform schemes nor for adaptive schemes with their higher order of convergence, which coincides with the result \eqref{eq: quadrature perturbed error}. In particular, the $p$-dependency of the constants appearing in Section~\ref{sec: PerturbationQuadrature} may not be severe.
    %
    \begin{table}[ht]
        \centering
        \addtolength{\tabcolsep}{-2pt}
        \begin{tabularx}{\textwidth}{l| X X X | X | X X X}
            & h1 & h2 & h3 & p & a1 & a2 & a3 \\ \hline 
            $q=p$     &  0.50 & 0.74 & 0.75 &   0.74  & 0.50 & 1.10 & 1.46\\ 
            $q=p+1$   &  0.50 &   0.75 &  0.76 &  0.74 & 0.50 & 1.08 & 1.48\\ 
            $q=p+2$   &  0.50 &   0.75 &  0.76 &  0.74 & 0.50 & 1.08 & 1.48\\ 
            $q=p+5$   &  0.50 &   0.75 &  0.76 &  0.74 & 0.50 & 1.08 & 1.48\\ 
            $q=p+11$  & 0.50 &  0.75 &  0.76 & 0.74  & 0.50 & 1.08 & 1.48
        \end{tabularx} 
        \caption{Experimental order of convergence of the error $|u-u_{hp,q}|_{H^1(\Omega)}$ with respect to the degrees of freedom for different number of quadrature points $q$. hj stands for uniform $h$-version with $p=j$, aj stands for adaptive $h$-version with $p=j$, and $p$ stands for uniform $p$-version}
        \label{tab:eoc_error}
    \end{table}
    \subsection{Quadrature error}
    The remaining numerical experiments are devoted to the quadrature error $|u_{hp}-u_{hp,q}|_{H^1(\Omega)}$ in accordance to Theorem~\ref{Quad Error}. In lack of an available $u_{hp}$ we approximate the quadrature error by $|u_{hp,p+11}-u_{hp,q}|_{H^1(\Omega)}$ for $q \in \{p,p+1,p+2,p+3,p+4,p+5\}$. That error versus the degrees of freedom is plotted in Figure~\ref{fig:quad_error} for the seven different discretization methods. The corresponding convergence rates are tabulated in Table~\ref{tab:eoc_quad_error}. Obviously, these are at least as large as the convergence rates for the total error $|u-u_{hp,q}|_{H^1(\Omega)}$, see Table~\ref{tab:eoc_error}, and increase with the number of quadrature points $q$. For large $q$ the computed quadrature error is dominated by the accumulation of rounding errors already for moderate degrees of freedom, see Figure~\ref{fig:quad_error}. The numerical experiments indicate that the results of Section~\ref{sec: PerturbationQuadrature} still hold even if $F_D$ is not affine, and could also hold for the uniform $p$-version.
    %
    \begin{table}[ht!]
        \centering
        \addtolength{\tabcolsep}{-2pt}
        \begin{tabularx}{\textwidth}{l|X X X|X|X X X}
            & h1 & h2 & h3 & p & a1 & a2 & a3 \\ \hline 
            $q=p$     &  0.70 &   0.72 &  0.71 & 1.04 & 0.50 & 1.10 & 1.44\\ 
            $q=p+1$   &  1.22 &   1.72 &  1.75 & 1.01 & 1.28 & 1.47 & 2.20\\ 
            $q=p+2$   &  1.73 &   2.33 &  2.56 & 0.99 & 1.78 & 2.09 & 2.61\\ 
            $q=p+3$   &  2.17 &   2.80 &  2.93 & 0.99 & 2.17 & 2.71 & 2.63\\ 
            $q=p+4$   &  2.46 &   3.19 &  3.25 & 1.08 & 2.44 & 2.82 & 2.80\\ 
            $q=p+5$   &  2.77 &   3.48 &  3.33 & --- & 2.78 & 2.97 & 2.44 
        \end{tabularx} 
        \caption{Experimental order of convergence of the quadrature error $|u_{hp,p+11}-u_{hp,q}|_{H^1(\Omega)} $ with respect to the degrees of freedom for different number of quadrature points $q$. hj stands for uniform $h$-version with $p=j$, aj stands for adaptive $h$-version with $p=j$, and $p$ stands for uniform $p$-version}
        \label{tab:eoc_quad_error}
    \end{table}
    \begin{figure}[htb!]
	    \hspace{-0.5cm}
	      \subfloat[$h$-unif.,~$p=1$]{
        \resizebox{0.25\textwidth}{!}{%
	    \begin{tikzpicture}[scale=0.7]
		      \begin{loglogaxis}[width=0.75\textwidth,mark size=3pt,line width=0.75pt,xmin=5e0,xmax=3e7,ymin=1e-14,ymax=1,legend style={at={(0,0)},anchor=south west}]

			    \addplot+[mark=o, color=blue]         table[x index=0,y index=9] {h1.txt};
 			    \addplot+[mark=x, color=red]     table[x index=0,y index=10] {h1.txt};
 			    \addplot+[mark=triangle, color=teal]     table[x index=0,y index=11] {h1.txt};
 			    \addplot+[mark=+, color=cyan]     table[x index=0,y index=12] {h1.txt};
 			    \addplot+[mark=square, color=magenta]     table[x index=0,y index=13] {h1.txt};
 			    \addplot+[mark=star, color=purple]     table[x index=0,y index=14] {h1.txt};
			
		      \end{loglogaxis}
	    \end{tikzpicture}}
        }%
	    \subfloat[$h$-unif.,~$p=2$]{ 
        \resizebox{0.25\textwidth}{!}{%
	    \begin{tikzpicture}[scale=0.7]
		      \begin{loglogaxis}[width=0.75\textwidth,mark size=3pt,line width=0.75pt,xmin=3e1,xmax=5e7,ymin=1e-14,ymax=1,legend style={at={(0,0)},anchor=south west}]
			    \addplot+[mark=o, color=blue]         table[x index=0,y index=9] {h2.txt};
 			    \addplot+[mark=x, color=red]     table[x index=0,y index=10] {h2.txt};
 			    \addplot+[mark=triangle, color=teal]     table[x index=0,y index=11] {h2.txt};
 			    \addplot+[mark=+, color=cyan]     table[x index=0,y index=12] {h2.txt};
 			    \addplot+[mark=square, color=magenta]     table[x index=0,y index=13] {h2.txt};
 			    \addplot+[mark=star, color=purple]     table[x index=0,y index=14] {h2.txt};
		      \end{loglogaxis}
	    \end{tikzpicture}}
        }%
	    \subfloat[$h$-unif.,~$p=3$]{
        \resizebox{0.25\textwidth}{!}{%
	    \begin{tikzpicture}[scale=0.7]
		      \begin{loglogaxis}[width=0.75\textwidth,mark size=3pt,line width=0.75pt,xmin=5e1,xmax=5e7,ymin=1e-14,ymax=1,	legend style={at={(0,0)},anchor=south west}]
			    \addplot+[mark=o, color=blue]         table[x index=0,y index=9] {h3.txt};
 			    \addplot+[mark=x, color=red]     table[x index=0,y index=10] {h3.txt};
 			    \addplot+[mark=triangle, color=teal]     table[x index=0,y index=11] {h3.txt};
 			    \addplot+[mark=+, color=cyan]     table[x index=0,y index=12] {h3.txt};
 			    \addplot+[mark=square, color=magenta]     table[x index=0,y index=13] {h3.txt};
 			    \addplot+[mark=star, color=purple]     table[x index=0,y index=14] {h3.txt};
		      \end{loglogaxis}
	    \end{tikzpicture}}
        }%
        \subfloat[$p$-unif.]{
        \resizebox{0.255\textwidth}{!}{%
	    \begin{tikzpicture}[scale=0.84]
		      \begin{loglogaxis}[width=0.8\textwidth,mark size=3pt,line width=0.75pt,xmin=4e1,xmax=1e6,ymin=1e-14,ymax=1,legend style={at={(0,0)},anchor=south west}	]
			    \addplot+[mark=o, color=blue]         table[x index=0,y index=9] {p.txt};
 			    \addplot+[mark=x, color=red]     table[x index=0,y index=10] {p.txt};
 			    \addplot+[mark=triangle, color=teal]     table[x index=0,y index=11] {p.txt};
 			    \addplot+[mark=+, color=cyan]     table[x index=0,y index=12] {p.txt};
 			    \addplot+[mark=square, color=magenta]     table[x index=0,y index=13] {p.txt};
 			    \addplot+[mark=star, color=purple]     table[x index=0,y index=14] {p.txt};
		      \end{loglogaxis}
	    \end{tikzpicture}}
        }%
        \\
        \centering
	    \subfloat[$h$-adap.,~$p=1$]{ 
	    \begin{tikzpicture}[scale=0.48]
		      \begin{loglogaxis}[width=0.61\textwidth,mark size=3pt,line width=0.75pt,xmin=1e1,xmax=2e7,ymin=1e-14,ymax=1,legend style={at={(0,0)},anchor=south west}	]

			    \addplot+[mark=o, color=blue]         table[x index=0,y index=9] {a1.txt};
 			    \addplot+[mark=x, color=red]     table[x index=0,y index=10] {a1.txt};
 			    \addplot+[mark=triangle, color=teal]     table[x index=0,y index=11] {a1.txt};
 			    \addplot+[mark=+, color=cyan]     table[x index=0,y index=12] {a1.txt};
 			    \addplot+[mark=square, color=magenta]     table[x index=0,y index=13] {a1.txt};
 			    \addplot+[mark=star, color=purple]     table[x index=0,y index=14] {a1.txt};
			
			
		      \end{loglogaxis}
	    \end{tikzpicture}}
	   \subfloat[$h$-adap.,~$p=2$]{
	    \begin{tikzpicture}[scale=0.48]
		      \begin{loglogaxis}[width=0.61\textwidth,mark size=3pt,line width=0.75pt,xmin=3e1,xmax=5e7,ymin=1e-14,ymax=6e2,legend style={at={(1,1)},anchor=north east}]

			    \addplot+[mark=o, color=blue]         table[x index=0,y index=9] {a2.txt};
 			    \addplot+[mark=x, color=red]     table[x index=0,y index=10] {a2.txt};
 			    \addplot+[mark=triangle, color=teal]     table[x index=0,y index=11] {a2.txt};
 			    \addplot+[mark=+, color=cyan]     table[x index=0,y index=12] {a2.txt};
 			    \addplot+[mark=square, color=magenta]     table[x index=0,y index=13] {a2.txt};
 			    \addplot+[mark=star, color=purple]     table[x index=0,y index=14] {a2.txt};
			
			    \legend{{$q=p$}, {$q=p+1$}, {$q=p+2$}, {$q=p+3$}, {$q=p+4$}, {$q=p+5$}}
		      \end{loglogaxis}
	    \end{tikzpicture}}
	   \subfloat[$h$-adap.,~$p=3$]{
	    \begin{tikzpicture}[scale=0.48]
		      \begin{loglogaxis}[width=0.61\textwidth,mark size=3pt,line width=0.75pt,xmin=1e2,xmax=1e7,ymin=1e-14,ymax=1,legend style={at={(0,0)},anchor=south west}	]

			    \addplot+[mark=o, color=blue]         table[x index=0,y index=9] {a3.txt};
 			    \addplot+[mark=x, color=red]     table[x index=0,y index=10] {a3.txt};
 			    \addplot+[mark=triangle, color=teal]     table[x index=0,y index=11] {a3.txt};
 			    \addplot+[mark=+, color=cyan]     table[x index=0,y index=12] {a3.txt};
 			    \addplot+[mark=square, color=magenta]     table[x index=0,y index=13] {a3.txt};
 			    \addplot+[mark=star, color=purple]     table[x index=0,y index=14] {a3.txt};
			

		      \end{loglogaxis}
	    \end{tikzpicture}}
	
	    \caption{Quadrature error $|u_{hp,p+11}-u_{hp,q}|_{H^1(\Omega)} $ vs.~degrees of freedom for different discretization methods. The legend is the same for all figures}
	    \label{fig:quad_error}
    \end{figure}
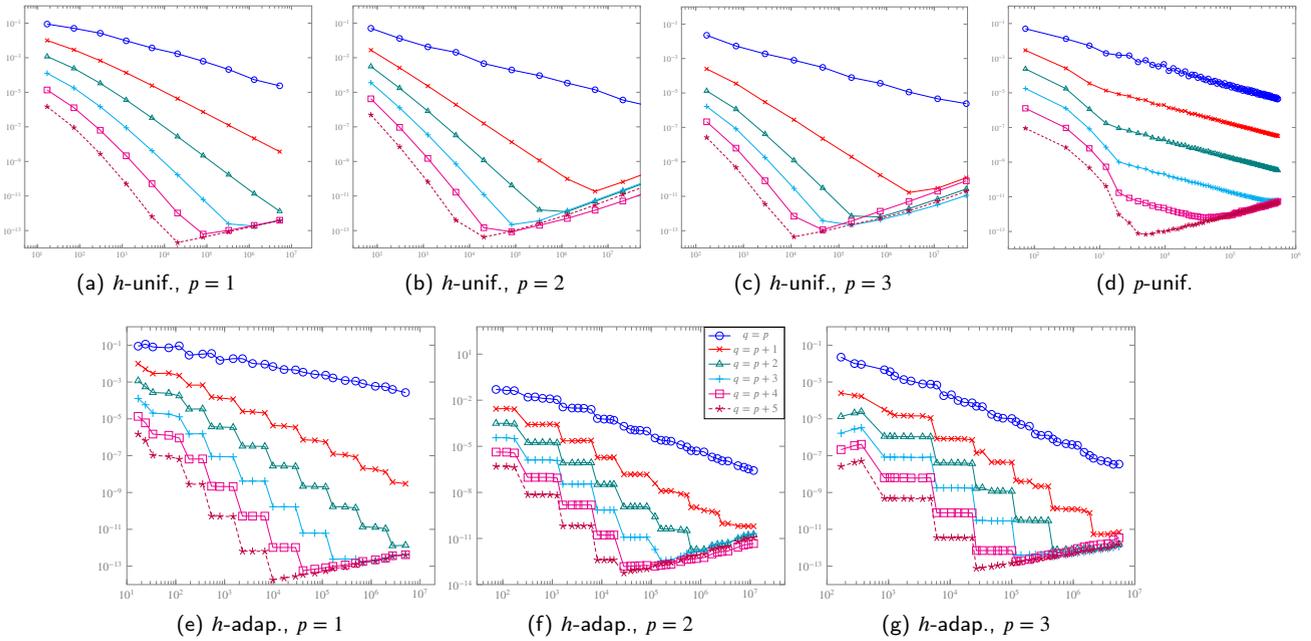

\section*{Acknowledgement}
The authors (Schönauer and Schröder) gratefully acknowledge the support by the Bundesministerium für Bildung, Wissenschaft und Forschung (BMBWF) under the Sparkling Science project SPA 01-080 'MAJA- Mathematische Algorithmen für Jedermann Analysiert'.

\printcredits

\section*{Declaration of interest}
    None.

\bibliographystyle{model1b-num-names}

\bibliography{references}

\end{document}